\documentclass{article}

\usepackage{algorithm,algpseudocode}
\usepackage[export]{adjustbox}
\usepackage{amsfonts}
\usepackage{amssymb}
\usepackage{amsmath}
\usepackage{amsthm} 
\usepackage[toc,page]{appendix}
\usepackage{bbm}
\usepackage{bm}
\usepackage{booktabs}
\usepackage{caption}
\usepackage{comment}
\usepackage{diagbox}
\usepackage{enumerate}
\usepackage{float}
\usepackage{graphicx}
\usepackage{hyperref} 
\usepackage{ifthen}
\usepackage[utf8]{inputenc}
\usepackage{mathtools}
\usepackage{multirow}
\usepackage{pifont} 
\usepackage{rotating}
\usepackage{stmaryrd}
\usepackage{subcaption}
\usepackage[normalem]{ulem}
\usepackage{verbatim}
\usepackage{xcolor}
\usepackage{xstring}

\usepackage{cleveref}
\usepackage{cancel}

\creflabelformat{enumi}{\textit{(#1)#2#3}}

\usepackage{pgf,tikz,pgfplots}
\usetikzlibrary{decorations.markings}
\pgfplotsset{compat=1.15}
\usepackage{mathrsfs}
\usetikzlibrary{arrows}

\newcommand{\rev}[1]{\textcolor{black}{#1}}

\newboolean{hidebool}
\setboolean{hidebool}{false}

\ifthenelse{\boolean{hidebool}}
  {
   \newcommand{\hide}[1]{}%
  }
  {
   \newcommand{\hide}[1]{#1}%
  }

\DeclareMathOperator{\tr}{tr}

\DeclareMathOperator{\rank}{rank}

\DeclareMathOperator{\range}{range}

\renewcommand*{\eqref}[1]{(\ref{#1})}

\newtheorem{theorem}{Theorem}[section]
\newtheorem{corollary}[theorem]{Corollary}
\newtheorem{lemma}[theorem]{Lemma}

\newtheorem{remark}{Remark}

\newtheorem{example}{Example}[section]
\makeatother

\newcommand\numberthis{\addtocounter{equation}{1}\tag{\theequation}}

\newcommand{\xmark}{\ding{55}}

\newcommand{\sbmat}[1]{\left[\begin{smallmatrix} #1 \end{smallmatrix}\right]} 

  \usepackage{nth}
  \usepackage{intcalc}

\title{Algorithm-agnostic low-rank approximation of operator monotone matrix functions\thanks{David Persson was supported by the SNSF research project \textit{Fast algorithms from low-rank updates}, grant number: 200020\_178806. Christopher Musco and Raphael A. Meyer were partially supported by NSF Award 2045590.}}

\author{David Persson\thanks{Institute of Mathematics, EPF Lausanne, Lausanne, Switzerland. \href{mailto:david.persson@epfl.ch}{david.persson@epfl.ch}} \and Raphael A. Meyer\thanks{New York University, New York, NY. \href{mailto:ram900@nyu.edu}{ram900@nyu.edu},\href{mailto:cmusco@nyu.edu}{cmusco@nyu.edu}} \and Christopher Musco\footnotemark[3]}

\begin{document}

\maketitle


\begin{abstract}
Low-rank approximation of a matrix function, $f(\bm{A})$, is an important task in computational mathematics. 
Most methods require direct access to $f(\bm{A})$, which is often considerably more expensive than accessing $\bm{A}$.
Persson and Kressner (SIMAX 2023) avoid this issue for symmetric positive semidefinite matrices by proposing funNyström, which first constructs a Nyström approximation to $\bm{A}$ using subspace iteration,  and then uses the approximation to directly obtain a low-rank approximation for $f(\bm{A})$. They prove that the method yields a near-optimal approximation whenever $f$ is a continuous operator monotone function with $f(0) = 0$.



We significantly generalize the results of Persson and Kressner beyond subspace iteration.
We show that if $\widehat{\bm{A}}$
is a near-optimal low-rank Nyström approximation to $\bm{A}$ then $f(\widehat{\bm{A}})$
is a near-optimal low-rank approximation to $f(\bm{A})$, \emph{independently of how $\widehat{\bm{A}}$ is computed}.
Further, we show sufficient conditions for a basis $\bm{Q}$ to produce a near-optimal Nyström approximation $\widehat{\bm{A}} = \bm{A}\bm{Q}(\bm{Q}^T \bm{A} \bm{Q})^{\dagger} \bm{Q}^T \bm{A}$. We use these results to establish that many common low-rank approximation methods produce near-optimal Nyström approximations to $\bm{A}$ and therefore to $f(\bm{A})$.

\end{abstract}
\section{Introduction}
\label{section:intro}


A common task in numerical linear algebra is to approximate a matrix function $f(\bm{A})$, or some property, like the trace or diagonal entries, of  $f(\bm{A})$.
Formally, given a symmetric matrix $\bm{A} \in \mathbb{R}^{n \times n}$ with eigendecomposition $\bm{A} = \bm{U} \bm{\Lambda} \bm{U}^T$ and a function $f$ defined on the eigenvalues of $\bm{A}$, $f(\bm{A})$ is defined as
\begin{equation*}
    f(\bm{A}) = \bm{U} f(\bm{\Lambda}) \bm{U}^T, \quad f(\bm{\Lambda}) = \begin{bmatrix} f(\lambda_1) & &\\ & \ddots & \\ & & f(\lambda_n) \end{bmatrix}.
\end{equation*}
Exactly computing $f(\bm{A})$ requires an eigendecomposition of $\bm{A}$, which is prohibitively expensive for large scale problems.
Instead, we may seek a cheaper approximation to $f(\bm{A})$ that suffices for a desired application.
In particular, recent research shows that low-rank approximations can help estimate quantities associated with \(f(\bm{A})\).
For example, \cite{saibabarandom, lizhu, saibaba} use a low-rank approximation to estimate $\tr(f(\bm{A}))$ for $f(x) = x$, $\log(1+x),$ and $\frac{1}{x+1}$, and
\cite{chen,jiang2021optimal,hutchpp,AHutchpp} use low-rank approximations of matrix functions as a variance reduction technique for the Girard-Hutchinson trace estimator \cite{cortinoviskressner,ubarusaad,ubaru2017applications}.
Similar techniques can be used to estimate the diagonal of $f(\bm{A})$ \cite{baston2022stochastic,reese2021bayesian}.


While low-rank approximations of $f(\bm{A})$ are clearly useful, a challenge in computing them is that most 
algorithms for low-rank approximation require at least some access to the matrix $f(\bm{A})$, for example in the form of matrix-vector products. This is true of all Krylov subspace methods, as well as randomized SVD methods based on sketching \cite{rsvd,Sarlos:2006,tropp2023randomized,tropp2017fixed,woodruff_sketching}. Since we cannot directly compute matrix-vector products with $f(\bm{A})$ (we do not have the matrix in hand) we need to resort to approximating them. This can be done, for example, with the Lanczos method \cite[Chapter 13]{functionsofmatrices} or rational Krylov subspace methods \cite{guttel}. However, such methods can  require many matrix-vector products with $\bm{A}$ to accurately compute a single matrix-vector product with $f(\bm{A})$. 
Consequently, the cost of obtaining a low-rank approximation to $f(\bm{A})$ can be significantly higher than obtaining a low-rank approximation to $\bm{A}$. 

\subsection{The Nystr\"om and funNystr\"om approximations}

Recent work in \cite{persson2022randomized} address the challenge discussed above by showing that, when $f$ is non-decreasing, it is often possible to entirely avoid the need to approximate matrix-vector products with $f(\bm{A})$ when computing a low-rank approximation. Indeed, constructing a low-rank Nystr\"om approximation to $\bm{A}$ often suffices to get a good low-rank approximation to $f(\bm{A})$. Before describing in detail the method of \cite{persson2022randomized}, which is called funNystr\"om, we briefly recall the general form of Nystr\"om approximation  \cite[Section 14]{randnla}. For a SPSD matrix $\bm{A} \in \mathbb{R}^{n \times n}$, the Nyström approximation with respect to a matrix $\bm{Q} \in \mathbb{R}^{n \times \ell}$ is
\begin{equation}\label{eq:nystrom}
    \widehat{\bm{A}}:= \bm{A} \bm{Q} (\bm{Q}^T \bm{A} \bm{Q})^{\dagger} \bm{Q}^T \bm{A},
\end{equation}
where $\dagger$ denotes the Moore-Penrose pseudoinverse.\footnote{In line with prior work \cite{rsvd,randnla,tropp2023randomized}, we use the term \emph{Nyström approximation} to refer to the general form $\bm{A
Q}(\bm{Q}^T\bm{AQ})^\dagger \bm{Q}^T\bm{A}$. We do not make any assumptions on how $\bm{Q}$ is generated. We note that early work on the Nyström method for low-rank approximation, including the seminal work in \cite{NIPS2000_19de10ad}, as well as many follow up papers \cite{DrineasMahoney:2005,MuscoMusco:2017,ZhangTsangKwok:2008}, use the term Nyström approximation more narrowly to refer to methods where $\bm{Q}$ is obtained by subsampling columns of the identity matrix (i.e. $\bm{AQ}$ is a subset of columns from $\bm{A}$).} 
The Nyström approximation depends only on $\range(\bm{Q})$ \cite[Proposition A.2]{tropp2017fixed}. 
Therefore, without loss of generality we assume that $\bm{Q}$ has orthonormal columns. 
\rev{A typical goal is to find $\bm{Q}$ such that $\|\bm{A} -\widehat{\bm{A}}\| \lesssim \|\bm{A} -{\bm{A}}_{(k)}\|$ for some $k\leq \ell$, where $\bm{A}_{(k)}$ denotes the optimal rank-$k$ approximation to $\bm{A}$, which can be computed via a truncated eigendecomposition for any unitarily invariant norm $\|\cdot \|$ (e.g., the Frobenius, nuclear, or operator norm).
It suffices for $\bm{Q}$ to span the top  $k$ eigenvectors of $\bm{A}$, so the typical goal is to efficiently find $\bm{Q}$ whose span approximately covers at least the top $k$ eigenvectors of $\bm{A}$.}
E.g., $\bm{Q}$ might be chosen to be an orthonormal basis for $\range(\bm{A}^{q-1} \bm{\Omega})$ or $\range\left(\begin{bmatrix} \bm{\Omega} & \bm{A}\bm{\Omega} & \ldots & \bm{A}^{q-1} \bm{\Omega} \end{bmatrix}\right)$ for some $q \geq 1$ and random sketching matrix $\bm{\Omega}$. 
Analyses for various choices of $\bm{Q}$ can be found in \cite{chen2022randomly,cortinovismaxvol,gittensmahoney,persson2022randomized,tropp2023randomized,tropp2017fixed,wang2019scalable}. 
\rev{Note that, when $\ell > k$, $\widehat{\bm{A}}$ could have rank $>k$.
In order to recover an exactly rank $k$ approximation, we can return the} best rank $k$ approximation to $\widehat{\bm{A}}$ instead of $\widehat{\bm{A}}$ itself. Denoted as $\widehat{\bm{A}}_{(k)}$, this matrix can be computed efficiently in $O(n\ell^2)$ time, plus the time to multiply $\bm{A}$ by $\ell$ vectors (see  \Cref{alg:nystrom}).

\begin{algorithm}
\caption{Nyström approximation}
\label{alg:nystrom}
\textbf{input:} SPSD $\bm{A} \in \mathbb{R}^{n \times n}$. Orthonormal basis $\bm{Q} \in \mathbb{R}^{n \times \ell}$. Rank $k$.\\
\textbf{output:} Rank $k$ approx. to $\bm{A}$ in factored form $\widehat{\bm{A}}_{(k)} = \widehat{\bm{U}}_{:,1:k} \widehat{\bm{\Lambda}}_{1:k,1:k}\widehat{\bm{U}}_{:,1:k}^T$. 
\begin{algorithmic}[1]
    \State Compute an eigendecomposition, $\bm{V} \bm{D} \bm{V}^T$, of $\bm{Q}^T \bm{AQ}$. \label{line:cholesky}
    \State Set $\bm{B} = \bm{AQ}\bm{V} (\bm{D}^{1/2})^{\dagger}$
    \State Compute the SVD of $\bm{B} = \widehat{\bm{U}} \bm{\Sigma} \bm{W}^T$. Set $\widehat{\bm{\Lambda}} = \bm{\Sigma}^2$ \label{line:lambda}
    \State \textbf{return} $\widehat{\bm{U}}_{:,1:k},\widehat{\bm{\Lambda}}_{1:k,1:k}$.
\end{algorithmic}
\end{algorithm}

The funNyström method of \cite{persson2022randomized} first forms a Nyström approximation $\widehat{\bm{A}}$, where $\bm{Q} = \range(\bm{A}^{q-1} \bm{\Omega})$ is an orthonormal basis constructed via subspace iteration.
The method\footnote{The work in \cite{persson2022randomized} considers a more restricted setting where $f(0) = 0$, in which case $f(\widehat{\bm{A}})_{(k)}$ simply equals $f(\widehat{\bm{A}})$. Throughout, we only make the weaker assumption that $f(0)\geq 0$.} then constructs a low-rank approximation to $f(\bm{A})$ by returning $f(\widehat{\bm{A}})_{(k)}$, which denotes the best $k$-rank approximation to $f(\widehat{\bm{A}})$. When $f$ is non-decreasing with $f(0)\geq 0$, and $\widehat{\bm{A}}_{(k)} = \widehat{\bm{U}}_{:,1:k}\widehat{\bm{\Lambda}}_{1:k,1:k}\widehat{\bm{U}}_{:,1:k}^T$ is available in factored form, as returned by \Cref{alg:nystrom}, the funNyström approximation can be computed as:
\begin{align}
    \label{eq:funnystrom}
    f(\widehat{\bm{A}})_{(k)} = \widehat{\bm{U}}_{:,1:k}f(\widehat{\bm{\Lambda}}_{1:k,1:k})\widehat{\bm{U}}_{:,1:k}^T. 
\end{align}
The intuition behind funNyström is that, if $f$ is a positive non-decreasing function, applying $f$ to an SPSD matrix $\bm{A}$ does not change the order of eigenvalues. As a result, if $\widehat{\bm{A}}_{(k)}$ was an \emph{optimal} rank $k$ approximation to $\bm{A}$ in any unitarily invariant norm, then $f(\widehat{\bm{A}})_{(k)}$ would be an optimal rank $k$ approximation to $f(\bm{A})$\rev{; see \cite[Lemma 1.1]{persson2022randomized}.}
Hence, we might hope that if $\widehat{\bm{A}}_{(k)}$ is a near-optimal rank $k$ approximation to $\bm{A}$, then $f(\widehat{\bm{A}})_{(k)}$ remains a near-optimal rank $k$ approximation to $f(\bm{A})$. 

Formally, \cite{persson2022randomized} derives error bounds for $\|f(\bm{A}) - f(\widehat{\bm{A}})_{(k)}\|$  for a variety of unitarily invariant norms for a special class of non-decreasing functions: \emph{operator monotone} functions satisfying $f(0) = 0$. 
A function $f$ is operator monotone if for any symmetric matrices $\bm{B}\succeq\bm{C}$ we have that $f(\bm{B}) \succeq f(\bm{C})$, \rev{where $\bm{B}\succeq\bm{C}$ means that $\bm{B}-\bm{C}$ is SPSD. }
Important examples used in applications include $\log(1+x)$, $x^\alpha$ for $\alpha \in [0,1]$, and the ``ridge function'' $\frac{x}{x+\lambda}$ for $\lambda \geq 0$ \cite{FrostigMuscoMusco:2016}.
Notably, all operator monotone functions are concave, so $\bm{A}$ has larger normalized eigenvalue gaps than $f(\bm{A})$. An important intuition behind the work of \cite{persson2022randomized} is that, given these larger gaps, subspace iteration run on $\bm{A}$ should converge no slower to the top eigenvectors of $f(\bm{A})$ than if we had run it on $f(\bm{A})$ directly.

\subsection{Our Contributions}
The main limitation of \cite{persson2022randomized} is that its analysis of the funNyström method is specialized to $\widehat{\bm{A}}$ produced by subspace iteration specifically, i.e. $\bm{Q}$ in \eqref{eq:nystrom} is an orthonormal basis for $\range(\bm{A}^{q-1} \bm{\Omega})$ where $\bm{\Omega}$ is a random matrix.
The literature contains a wide variety of other low-rank approximation algorithms, many of which often outperform subspace iteration, including e.g. randomized Krylov methods \cite{MeyerMuscoMusco:2024,MM15}, methods based on sparse or structured random projections \cite{ClarksonWoodruff:2013,LibertyWoolfeMartinsson:2007}, frequent directions methods \cite{GhashamiLibertyPhillips:2016}, or methods based on random sampling \cite{chen2022randomly,DrineasMahoney:2005, DrineasMahoneyMuthukrishnan:2008}. The analysis in \cite{persson2022randomized} tells us nothing about how good a low-rank approximation  $f(\widehat{\bm{A}})_{(k)}$ would be if $\widehat{\bm{A}}$ is produced by such methods.

We address this gap by significantly generalizing the results in \cite{persson2022randomized}\rev{, proving results that hold \textit{independently of how $\widehat{\bm{A}}$ was computed}}. \rev{In particular, for the nuclear, Frobenius, and operator norms,} we show that if $\widehat{\bm{A}}$ is \emph{any} near-optimal low-rank approximation to $\bm{A}$ so that $\bm{A} \succeq \widehat{\bm{A}} \succeq \bm{0}$ \rev{(see, e.g., \eqref{eq:nearopt})}, then for any positive, continuous operator monotone function, $f(\widehat{\bm{A}})_{(k)}$ is a near-optimal low-rank approximation to $f(\bm{A})$. 
Importantly, \rev{as a consequence of \cite[Lemma 1]{gittensmahoney}} any Nyström approximation satisfies $\bm{A} \succeq \widehat{\bm{A}} \succeq \bm{0}$,\footnote{\rev{The error satisfies $\bm{A} - \widehat{\bm{A}} = \bm{A}^{1/2}(\bm{I} - \bm{P}_{\bm{A}^{1/2} \bm{Q}}) \bm{A}^{1/2} \succeq \bm{0}$, where $\bm{P}_{\bm{A}^{1/2} \bm{Q}}$ denotes the orthogonal projector onto $\range(\bm{A}^{1/2} \bm{Q})$.}} so our guarantees extend to all possible Nyström approximations to $\bm{A}$, no matter how $\bm{Q}$ is obtained.
For example, \rev{for the nuclear norm, we prove} the following.
\begin{theorem}\label{theorem:nuclear_black_box_intro}
    Suppose that $\bm{A} \succeq \widehat{\bm{A}} \succeq \bm{0}$. Let $\bm{A}_{(k)}$ and $\widehat{\bm{A}}_{(k)}$ be optimal rank $k$ approximations to $\bm{A}$ and $\widehat{\bm{A}}$, respectively. Further suppose that for $\varepsilon \geq 0$,
    \begin{equation}\label{eq:nearopt}
        \|\bm{A}-\widehat{\bm{A}}_{(k)}\|_* \leq (1+\varepsilon)\|\bm{A}-\bm{A}_{(k)}\|_*,
    \end{equation}
    where $\|\cdot\|_*$ denotes the nuclear norm. Then for any continuous operator monotone function $f:[0,\infty) \to [0,\infty)$ we have
    \begin{equation*}
        \|f(\bm{A})-f(\widehat{\bm{A}})_{(k)}\|_* \leq (1+\varepsilon)\|f(\bm{A})-f(\bm{A})_{(k)}\|_*.
    \end{equation*}
\end{theorem}
The above theorem applies to the nuclear norm $\|\cdot \|_*$, however we also prove similar results when low-rank approximation error is measured in the the operator norm (\Cref{theorem:spectral_black_box}) and the Frobenius norm (\Cref{theorem:frobenius_black_box}).


Further, to facilitate the application of these theorems, we present sufficient conditions for an orthonormal basis $\bm{Q}$ to produce a near-optimal Nyström approximation of the form \eqref{eq:nystrom}. For example, we prove the following.

\begin{theorem}\label{theorem:nuclear_grey_box_intro}
Let $\bm{A} \succeq \bm{0}$ and let $\bm{Q}$ be an orthonormal basis so that, for $\varepsilon \geq 0$,
    \begin{equation}\label{eq:introeq}
        \|\bm{A} - (\bm{Q} \bm{Q}^T \bm{A})_{(k)}\|_* \leq (1+\varepsilon) \|\bm{A}-\bm{A}_{(k)}\|_*.
    \end{equation}
    Then if $\widehat{\bm{A}} = \bm{A} \bm{Q} (\bm{Q}^T \bm{A} \bm{Q})^{\dagger} \bm{Q}^T \bm{A}$ we have
\begin{equation*}
    \|\bm{A}-\widehat{\bm{A}}_{(k)}\|_* \leq (1+\varepsilon)\|\bm{A}-\bm{A}_{(k)}\|_*. 
\end{equation*}
\end{theorem}
Similar guarantees 
are proven in \cite{tropp2023randomized}, but they are constrained to the case when $\bm{Q}$ has exactly $k$ columns.
In contrast, our results allow $\bm{Q}$ to have more than $k$ columns.
Guarantees of the form \eqref{eq:introeq} are common in the literature; see e.g. \cite{aineshlowrank,bakshi2023krylov,iyer2018iterative,MeyerMuscoMusco:2024,MM15}.
\Cref{theorem:nuclear_grey_box_intro} allows us to translate these existing results into guarantees for the Nyström approximation.
Then, by using \Cref{theorem:nuclear_black_box_intro}, these existing results for $\bm{Q}$ translate all the way into results for the funNyström approximation. 
\rev{For example, {\cite[Theorem 5.1]{aineshlowrank}} shows that there exists an algorithm which performs $O\left(\frac{k\log(n/\varepsilon)}{\varepsilon^{1/3}}\right)$ matrix-vector products with $\bm{A}$, where $\varepsilon \in (0,1)$, and returns an orthonormal basis $\bm{Q} \in \mathbb{R}^{n \times k}$ such that with probability at least $0.9$}
    \[
        \rev{\|\bm{A} - \bm{Q}\bm{Q}^T\bm{A}\|_* \leq (1+\varepsilon) \|\bm{A} - \bm{A}_{(k)}\|_*.}
    \]
\rev{By using \Cref{theorem:nuclear_black_box_intro,theorem:nuclear_grey_box_intro}, we immediately get the following corollary.}

\begin{corollary}
    Let $\bm{A} \succeq \bm{0}$ and $\varepsilon \in (0,1)$.
    Let $f:[0,\infty) \to [0,\infty)$ be a continuous operator monotone function.
    There is an algorithm which performs $\rev{O\left(\frac{k\log(n/\varepsilon)}{\varepsilon^{1/3}}\right)}$ matrix-vector products with $\bm{A}$ and returns a rank $k$ Nyström approximation $\widehat{\bm{A}}$ such that with probability at least $0.9$
    \[
        \|f(\bm{A}) - f(\widehat{\bm{A}})_{(k)}\|_* \leq (1+\varepsilon) \|f(\bm{A}) - f(\bm{A})_{(k)}\|_*.
    \]
\end{corollary}
\rev{Alternatively, we can recover a main result from the original funNyström paper \cite{persson2022randomized}. By a simple modification of the proof of \cite[Theorem 4.1]{tropp2017fixed} one can derive an expectation bound on the error $\|\bm{A} - \widehat{\bm{A}}_{(k)}\|_*$ when the Nyström approximation $\widehat{\bm{A}}$ is constructed using subspace iteration. Then, by using \Cref{theorem:nuclear_black_box_intro} we immediately recover \cite[Theorem 3.9]{persson2022randomized}. }

\begin{corollary}
    \rev{Let $\bm{A} \succeq \bm{0}$ and let $\gamma = \lambda_{k+1}/\lambda_k$ denote the $k^{\text{th}}$ spectral gap of $\bm{A}$. Let $\bm{Q}$ be an orthonormal basis for $\range(\bm{A}^{q-1} \bm{\Omega})$, where $q \geq 1$ and $\bm{\Omega}$ is a $n \times (k+p)$ random matrix whose entries are independent identically distributed $N(0,1)$ random variables. If $\widehat{\bm{A}} = \bm{A} \bm{Q}(\bm{Q}^T \bm{A} \bm{Q})^{\dagger} \bm{Q}^T \bm{A}$ and $p \geq 2$, then for any continuous operator monotone function $f:[0,\infty) \to [0,\infty)$ we have}
    \begin{equation*}
        \rev{\mathbb{E}\|f(\bm{A})-f(\widehat{\bm{A}})_{(k)}\|_* \leq \left(1+ \gamma^{2(q-1)} \frac{k}{p-1}\right)\|f(\bm{A}) - f(\bm{A})_{(k)}\|_*.} 
    \end{equation*}
\end{corollary}
As for \Cref{theorem:nuclear_black_box_intro}, we give similar guarantees to \Cref{theorem:nuclear_grey_box_intro} in the Frobenius and operator norm.
We also provide guarantees for eigenvalue estimation. 
As in \cite{persson2022randomized}, a strength of our results is that one does not need to know $f$ beforehand. The user can compute and store a Nyström approximation for $\bm{A}$ and use it to construct a funNyström approximation to \emph{any} continuous operator monotone function, and our results show that this approximation is always near-optimal as long as the Nyström approximation is near-optimal. 

In addition to our main theorems relating Nystr\"om approximation of $\bm{A}$ to funNystr\"om approximation of $f(\bm{A})$, we investigate if we can relax any of the assumptions in the results.
Specifically, we ask: $(i)$ is the assumption that $\bm{A} \succeq \widehat{\bm{A}}$ necessary? $(ii)$ is the operator monotonicity assumption necessary, or is it sufficient to assume that $f$ is concave and increasing? For some norms, we are able to find counterexamples that show that these assumptions are in fact necessary. Our findings are summarized in \Cref{table:summ}. We note, however, that our counterexamples only show that there are cases when $\|\bm{A}-\widehat{\bm{A}}_{(k)}\| \leq (1+\varepsilon)\|\bm{A}-\bm{A}_{(k)}\|$, but $\|f(\bm{A})-f(\widehat{\bm{A}})_{(k)}\| > (1+\varepsilon + \delta)\|f(\bm{A})-f(\bm{A})_{(k)}\|$ for some non-zero $\delta$. They do not, for example, rule out the possibility that $\|\bm{A}-\widehat{\bm{A}}_{(k)}\| \leq (1+\varepsilon)\|\bm{A}-\bm{A}_{(k)}\|$ could  imply that $\|f(\bm{A})-f(\widehat{\bm{A}})_{(k)}\| \leq (1+C\varepsilon)\|f(\bm{A})-f(\bm{A})_{(k)}\|$ for some fixed constant $C \geq 1$ (e.g., $C=2$ or $C=10$). Finding stronger counter examples is an interesting direction for future work.
\begin{table}[H]
\centering
\caption{
    \rev{Our current knowledge of what assumptions are required for $\|\bm{A}-\widehat{\bm{A}}_{(k)}\| \leq (1+\varepsilon)\|\bm{A}-\bm{A}_{(k)}\|$ to imply $\|f(\bm{A})-f(\widehat{\bm{A}})_{(k)}\| \leq (1+\varepsilon)\|f(\bm{A})-f(\bm{A})_{(k)}\|$.
	A checkmark (\checkmark) indicates that the implication holds under the listed assumptions, and we provide a link to the theorem proving it.
	A cross (\xmark) indicates that a counterexample with the listed assumptions violates the guarantee.
    The question marks (?) indicate open problems.
	The asterisk for the Frobenius norm denotes that we make a slightly stronger assumption for 
    that guarantee; see \Cref{section:frobenius_nuclear_black_box} for more details. See \Cref{section:eigenvalue_blackbox} for a formal description of eigenvalue approximation. }}
\begin{tabular}{@{}lcccc@{}} \toprule
	& \multicolumn{2}{c}{$f$ is operator monotone} & \multicolumn{2}{c}{$f$ is concave, non-decreasing} \\ \cmidrule(lr){2-3} \cmidrule(l){4-5}
	Guarantee
		& \(\bm{A} \succeq \widehat{\bm{A}}\)
		& \(\bm{A} \cancel{\succeq} \widehat{\bm{A}}\)
		& \(\bm{A} \succeq \widehat{\bm{A}}\)
		& \(\bm{A} \cancel{\succeq} \widehat{\bm{A}}\) \\ \midrule
	Nuclear Norm
		& \(\checkmark_{\ref{theorem:nuclear_black_box}}\)
		& \(\text{\xmark}_{\ref{eg:psd-needed-for-frob-nuc}}\)
		& \(\text{\xmark}_{\ref{eg:mono-needed-for-nuc}}\)
		& \(\text{\xmark}_{\ref{eg:psd-needed-for-frob-nuc}}\) \\
	Frobenius Norm
		& \(\checkmark^*_{\ref{theorem:frobenius_black_box}}\)
		& \(\text{\xmark}_{\ref{eg:psd-needed-for-frob-nuc}}\)
		& ?
		& \(\text{\xmark}_{\ref{eg:psd-needed-for-frob-nuc}}\) \\
	\rev{Operator} Norm
		& \(\checkmark_{\ref{theorem:spectral_black_box}}\)
		& \(\checkmark_{\ref{theorem:spectral_black_box}}\)
		& ?
		& \(\text{\xmark}_{\ref{eg:mono-needed-for-spectral}}\) \\
	\rev{Eigenvalue Approx.}
		& \(\checkmark_{\ref{theorem:eigenvalue_black_box}}\)
		& \(\checkmark_{\ref{theorem:eigenvalue_black_box}}\)
		& \(\checkmark_{\ref{theorem:eigenvalue_black_box}}\)
		& \(\checkmark_{\ref{theorem:eigenvalue_black_box}}\) \\ \bottomrule
\end{tabular}
\label{table:summ}
\end{table}

\subsection{Notation}

We consider a symmetric positive semidefinite (SPSD) matrix $\bm{A}$.
$\widehat{\bm{A}}$ usually denotes an approximation to $\bm{A}$ so that $\bm{A} \succeq \widehat{\bm{A}} \succeq \bm{0}$.
For a matrix $\bm{B}$, $\sigma_i(\bm{B})$ denotes the $i^{\text{th}}$ largest singular value \rev{and when $\bm{B}$ is square $\lambda_i(\bm{B})$ denotes the $i^{\text{th}}$ largest (in magnitude) eigenvalue of $\bm{B}$. If $\bm{B} \succeq \bm{0}$ we have $\lambda_i(\bm{B}) = \sigma_i(\bm{B})$.} For conciseness, we drop the argument and let $\lambda_i$ denote $\lambda_i(\bm{A})$ and $\widehat{\lambda}_i$ denote $\lambda_i(\widehat{\bm{A}})$\rev{, when there is no risk of confusion.}
For a matrix $\bm{B}$ we let $\bm{P}_{\bm{B}}$ denote orthogonal projector onto $\range(\bm{B})$.
Hence, if $\bm{V}$ is an orthonormal basis for $\range(\bm{B})$ we have $\bm{P}_{\bm{B}} = \bm{V} \bm{V}^T$.
We let $\bm{B}_{(k)}$ denote the best rank $k$ approximation to $\bm{B}$ obtained by the truncated singular value decomposition.
For $p \in [1,\infty]$ we denote by $\|\cdot\|_{(p)}$ the Schatten $p$-norm. We write $\|\cdot\|_{2} = \|\cdot\|_{(\infty)}$, $\|\cdot\|_{F} = \|\cdot\|_{(2)}$, and $\|\cdot\|_* = \|\cdot\|_{(1)}$ to denote the operator, Frobenius, and nuclear norm respectively. \rev{Recall that the operator norm is sometimes called the \emph{spectral norm}.}

\section{Good Nyström approximations imply good funNyström approximations}\label{section:black_box}
In this section prove our main results: that near\rev{-}optimal Nyström approximations imply near\rev{-}optimal funNyström approximations for the nuclear, Frobenius, and operator norms.
We also prove analogous guarantees for eigenvalue estimation. 
We begin with establishing a few useful lemmas. We start by recalling some basic properties of operator monotone and concave functions. 
\begin{lemma}\label{lemma:opmon_properties}
    Let $f:[0,\infty) \to [0,\infty)$ be a continuous operator monotone function. Then,
    \begin{enumerate}[(i)]
        \item $f$ is concave; \label{enum:opmon_implies_opconcave}
        \item $f \in C^{\infty}(0,\infty)$. \label{eqnum:opmon_implies_smooth}
    \end{enumerate}
\end{lemma}
\begin{proof}
(\textit{\ref*{enum:opmon_implies_opconcave}}) This is due to \cite[Theorem V.2.5]{bhatia}.
(\textit{\ref*{eqnum:opmon_implies_smooth}}) This is due to \cite[p.134-135]{bhatia}. 
\end{proof}
By the above, the following \rev{simple facts} about concave functions also extends to continuous operator monotone functions.

\begin{lemma}\label{lemma:concave_properties}
    Let $f:[0,\infty) \to [0,\infty)$ be a concave function. Then,
    \begin{enumerate}[(i)]
        \item $\frac{f(x)}{x}$ is decreasing for \(x > 0\); \label{enum:decreasing_ratio}
        \item $f(tx) \leq tf(x)$ for $t \geq 1$;\label{enum:t>1}
        \item $f(tx) \geq tf(x)$ for $0 \leq t \leq 1$\rev{;} \label{enum:t<1}
        \item \rev{For $t \geq 0$, the function $f(x) - f(x-t)$ is decreasing. \label{enum:decreasing_difference}}
    \end{enumerate}
\end{lemma}
\begin{proof}
    \rev{(\textit{\ref*{enum:decreasing_ratio}})-(\textit{\ref*{enum:t<1}}) are straighforward to prove; e.g., see the answers in \cite{f(x)/x}, \cite{t>1}, and \cite{t<1} respectively. (\textit{\ref*{enum:decreasing_difference}}) is a simple consequence of the fact that $R(x_1,x_2) = \frac{f(x_2) - f(x_1)}{x_2 - x_1}$ is decreasing in $x_2$ for any fixed $x_1$ (or vice versa) \cite[p.431]{farwig}. Hence, $R(y-t,y) \geq R(x-t,x)$ for $x \geq y$, which yields the desired result. }



\end{proof}

The following lemma provides an upper bound for the Schatten norm difference of two ordered SPSD matrices. 
In particular, it bounds $\|\bm{B}-\bm{C}\|_{(p)}^p$ by the difference of the Schatten norms of $\bm{B}$ and $\bm{C}$.
The latter is easier to analyze because it allows us to separate the eigenvalues of $\bm{B}$ and $\bm{C}$.
This will be especially useful for proving results for the nuclear and Frobenius norm.
\begin{lemma}\label{lemma:schatten_difference}
    Let $\bm{B} \succeq \bm{C} \succeq \bm0$, then for $p \geq 1$
    \begin{equation*}
        \|\bm{B} - \bm{C}\|_{(p)} \leq \left(\|\bm{B}\|_{(p)}^p - \|\bm{C}\|_{(p)}^p\right)^{1/p}.
    \end{equation*}
    When $p = 1$ the inequality becomes an equality. 
\end{lemma}
\begin{proof}
By a result by McCarthy \cite[Lemma 2.6]{mccarthy} we know that if $\bm{X},\bm{Y} \succeq \bm{0}$,
\begin{equation*}
    \tr((\bm{X} + \bm{Y})^p) \geq \tr(\bm{X}^p) + \tr(\bm{Y}^p).
\end{equation*}
Setting $\bm{X} = \bm{B} - \bm{C}$ and $\bm{Y} = \bm{C}$ yields the desired inequality. 
When $p=1$, using the fact that $\bm{B} \succeq \bm{C} \succeq \bm{0}$, we find that
\begin{equation*}
    \|\bm{B} - \bm{C}\|_* = \tr(\bm{B} - \bm{C}) = \tr(\bm{B}) - \tr(\bm{C}) = \|\bm{B}\|_* - \|\bm{C}\|_*,
\end{equation*}
\rev{as required.}
\end{proof}
    

\subsection{\rev{Nyström to funNyström:} Frobenius and nuclear norm guarantees}\label{section:frobenius_nuclear_black_box}
In this section we prove two results.
First, we prove \Cref{theorem:nuclear_black_box_intro}.
\begin{theorem}[\Cref{theorem:nuclear_black_box_intro} restated]
\label{theorem:nuclear_black_box}
    Suppose that, for $\varepsilon \geq 0$, $\bm{A} \succeq \widehat{\bm{A}} \succeq \bm{0}$ satisfy
    \begin{equation*}
        \|\bm{A}-\widehat{\bm{A}}_{(k)}\|_* \leq (1+\varepsilon)\|\bm{A}-\bm{A}_{(k)}\|_*.
    \end{equation*}
    Then for any continuous operator monotone function $f:[0,\infty) \to [0,\infty)$,
    \begin{equation*}
        \|f(\bm{A})-f(\widehat{\bm{A}})_{(k)}\|_* \leq (1+\varepsilon)\|f(\bm{A})-f(\bm{A})_{(k)}\|_*.
    \end{equation*}
\end{theorem}

Additionally, we prove an analogous guarantee for the Frobenius norm.
\begin{theorem}\label{theorem:frobenius_black_box}
    Suppose that, for $\varepsilon \geq 0$, $\bm{A} \succeq \widehat{\bm{A}} \succeq \bm{0}$ satisfy
    \begin{equation*}
        \|\bm{A}\|_F^2-\|\widehat{\bm{A}}_{(k)}\|_F^2 \leq (1+\varepsilon)\|\bm{A}-\bm{A}_{(k)}\|_F^2.
    \end{equation*}
    Then for any continuous operator monotone function $f:[0,\infty) \to [0,\infty)$,
    \begin{equation*}
        \|f(\bm{A})-f(\widehat{\bm{A}})_{(k)}\|_F^2 \leq (1+\varepsilon)\|f(\bm{A})-f(\bm{A})_{(k)}\|_F^2.
    \end{equation*}
\end{theorem}
We note that, by \Cref{lemma:schatten_difference}, $\|\bm{A}\|_F^2 - \|\widehat{\bm{A}}_{(k)}\|_F^2 \leq (1+\varepsilon) \|\bm{A} - \bm{A}_{(k)}\|_F^2$ is a \emph{stronger} assumption than $\|\bm{A} - \widehat{\bm{A}}\|_F^2 \leq (1+\varepsilon) \|\bm{A} - \bm{A}_{(k)}\|_F^2$. I.e., \Cref{theorem:frobenius_black_box} requires assuming more than that $\widehat{\bm{A}}_{(k)}$ is a near\rev{-}optimal low-rank approximation to $\bm{A}$ in the Frobenius norm. However, the assumption is still reasonable because, as we show in Section~\ref{section:gray_box}, many standard low-rank algorithms return results that satisfy this stronger guarantee.

We prove \Cref{theorem:nuclear_black_box,theorem:frobenius_black_box} as special cases of a single general theorem about Schatten norms.
In particular, by \Cref{lemma:schatten_difference} we know that $\|\bm{A}\|_* - \|\widehat{\bm{A}}\|_* = \|\bm{A} - \widehat{\bm{A}}\|_*$, so the following theorem about Schatten norms immediately implies both \Cref{theorem:nuclear_black_box,theorem:frobenius_black_box} by taking \(p=1\) and \(p=2\), respectively.


\begin{theorem}\label{theorem:schatten_black_box}
    Fix $p \in [1,\infty)$. Suppose that, for $\varepsilon \geq 0$, $\bm{A} \succeq \widehat{\bm{A}} \succeq \bm{0}$ satisfy
    \begin{equation*}
        \|\bm{A}\|_{(p)}^p-\|\widehat{\bm{A}}_{(k)}\|_{(p)}^p \leq (1+\varepsilon)\|\bm{A}-\bm{A}_{(k)}\|_{(p)}^{p}
    \end{equation*}
   Then for any continuous operator monotone function $f:[0,\infty) \to [0,\infty)$,
    \begin{equation*}
        \|f(\bm{A})-f(\widehat{\bm{A}})_{(k)}\|_{(p)}^p \leq (1+\varepsilon)\|f(\bm{A})-f(\bm{A})_{(k)}\|_{(p)}^p.
    \end{equation*}
\end{theorem}



\begin{proof}[Proof of \Cref{theorem:schatten_black_box}]
    Let $\lambda_i$ and $\widehat{\lambda_i}$ denote the $i^\text{th}$ largest eigenvalues of $\bm{A}$ and $\widehat{\bm{A}}$, respectively. Our assumption on $\widehat{\bm{A}}$ implies that
\begin{equation}\label{eq:tracediff}
    \sum\limits_{i=1}^k (\lambda_i^p - \widehat{\lambda}_i^p) \leq \varepsilon \sum\limits_{i=k+1}^n \lambda_i^p.
\end{equation}
By Weyl’s monotonicity principle, $\bm{A} \succeq \widehat{\bm{A}}$ implies that $\lambda_i \geq \widehat{\lambda}_i$ \cite[Corollary 7.7.4 (c)]{matrixanalysis}.
Let us define $(1-\delta_i) = \left(\frac{\widehat{\lambda}_i}{\lambda_i}\right)^p$ for $\delta_i \in [0,1]$ for $i = 1,\ldots,k$.
Hence, \eqref{eq:tracediff} implies that
\begin{equation}\label{eq:tracediff2}
    \sum\limits_{i=1}^k \delta_i \lambda_i^p \leq \varepsilon \sum\limits_{i=k+1}^n \lambda_i^p.
\end{equation}

By \Cref{lemma:concave_properties} (\textit{\ref*{enum:t<1}}), we know that \(f(\widehat{\lambda}_i)^p \geq (1-\delta_i)f(\lambda_i)^p\), and so
\begin{equation}\label{eq:nuclear_upperbound}
    \sum\limits_{i=1}^k (f(\lambda_i)^p - f(\widehat{\lambda}_i)^p)
    \leq \sum\limits_{i=1}^k \delta_if(\lambda_i)^p.
\end{equation}
\Cref{lemma:concave_properties} (\textit{\ref*{enum:decreasing_ratio}}) also shows that, for all $i = 1,\ldots,k$ and \(j = k+1,\ldots, n\), $\frac{\lambda_{j}}{\lambda_i} \leq \frac{f(\lambda_{j})}{f(\lambda_i)}$. \rev{For all such pairs $(i,j)$ we have $\frac{\lambda_{j}^p}{\delta_i\lambda_i^p} \leq \frac{f(\lambda_{j})^p}{\delta_if(\lambda_i)^p}$. By summing over all $j$ we get
\begin{equation*}
    \frac{\sum\limits_{j=k+1}^n\lambda_{j}^p}{\delta_i\lambda_i^p} \leq \frac{\sum\limits_{j=k+1}^nf(\lambda_{j})^p}{\delta_if(\lambda_i)^p}.
\end{equation*}
Equivalently,
\begin{equation}\label{eq:nuclear_gap}
    \delta_i f(\lambda_i)^p \leq \frac{\delta_i \lambda_i^p}{\sum\limits_{j=k+1}^n\lambda_{j}^p} \sum\limits_{j=k+1}^nf(\lambda_{j})^p.
\end{equation}}
Combining \eqref{eq:nuclear_gap} with \eqref{eq:nuclear_upperbound} and \eqref{eq:tracediff2} gives
\begin{align*}
    \sum\limits_{i=1}^k (f(\lambda_i)^p - f(\widehat{\lambda}_i)^p)
    &\leq \sum_{i=1}^k \delta_i f(\lambda_i)^p \leq \frac{\sum\limits_{i=1}^k\delta_i\lambda_i^p}{\sum\limits_{j=k+1}^n\lambda_{j}^p} \sum\limits_{j=k+1}^nf(\lambda_{j})^p \leq \varepsilon \sum\limits_{i=k+1}^n f(\lambda_i)^p.
    \numberthis \label{eq:nuclear_desired_inequality}
\end{align*}
Hence, \eqref{eq:nuclear_desired_inequality} implies that
\begin{equation*}
    \|f(\bm{A})\|_{(p)}^p - \|f(\widehat{\bm{A}})_{(k)}\|_{(p)}^p \leq (1+\varepsilon)\|f(\bm{A}) - f(\bm{A})_{(k)}\|_{(p)}^p.
\end{equation*}
We have $f(\bm{A}) \succeq f(\widehat{\bm{A}})_{(k)} \succeq \bm{0}$ since $f$ is operator monotone. The desired inequality follows from applying \Cref{lemma:schatten_difference}.
\end{proof}

\subsection{\rev{Nyström to funNyström:} Operator norm guarantees}

We next \rev{prove} an analogue to \Cref{theorem:nuclear_black_box,theorem:frobenius_black_box} for the operator norm. Our operator norm result is actually more general since we do not require access to an approximation $\widehat{\bm{A}}$ satisfying $\widehat{\bm{A}} \preceq \bm{A}$. The result applies to \emph{any} $\bm{B}$ such that $\|\bm{A} - \bm{B}\|_2 \leq (1+\varepsilon)\|\bm{A} - \bm{A}_{(k)}\|_2$.
We do not even  require that $\bm{B}$ is rank $k$.

\begin{theorem}\label{theorem:spectral_black_box}
    Suppose that, for $\varepsilon \geq 0$, $\bm{A}, \bm{B} \succeq \bm{0}$ satisfy
    \begin{equation*}
        \|\bm{A}-\bm{B}\|_2 \leq (1+\varepsilon) \|\bm{A} - \bm{A}_{(k)}\|_2. 
    \end{equation*}
     Let $r = \rank(\bm{B})$. For any continuous operator monotone function $f:[0,\infty) \to [0,\infty)$,
    \begin{equation*}
        \|f(\bm{A})-f(\bm{B})_{(r)}\|_2 \leq (1+\varepsilon) \|f(\bm{A}) - f(\bm{A})_{(k)}\|_2. 
    \end{equation*}
\end{theorem}
\begin{proof}
First assume that $f(0) = 0$. As a consequence of this assumption, we have $f(\bm{B})_{(r)} = f(\bm{B})$. We leverage \cite[Theorem X.1.1]{bhatia} about operator monotone functions, which says that $\|f(\bm{A})-f(\bm{B})\|_2 \leq f(\|\bm{A}-\bm{B}\|_2)$.
Then, recalling from \Cref{lemma:opmon_properties} that \(f\) is increasing and concave, we have that
\begin{align*}
    \|f(\bm{A})-f(\bm{B})\|_2
    &\leq f(\|\bm{A}-\bm{B}\|_2) \tag{\cite[Theorem X.1.1]{bhatia}} \\
    &\leq f((1+\varepsilon)\lambda_{k+1}) \\
    &\leq (1+\varepsilon) f(\lambda_{k+1}) \tag{\Cref{lemma:concave_properties} (\textit{\ref*{enum:t>1}})} \\
    &= (1+\varepsilon)\|f(\bm{A}) - f(\bm{A})_{(k)}\|_2,
\end{align*}
which yields the desired inequality for the case when $f(0) = 0$. 
We now consider the general case when $f(0) \geq 0$. Write $f(x) = g(x) + f(0)$, where $g(x)$ is a continuous operator monotone function satisfying $g(0) = 0$. 
Let $\bm{P}_{\bm{B}}$ be the orthogonal projector onto $\range(\bm{B})$. We have that $f(\bm{B})_{(r)} = g(\bm{B}) + f(0)\bm{P}_{\bm{B}}$, so
\begin{equation}\label{eq:sum}
    f(\bm{A}) - f(\bm{B})_{(r)} = g(\bm{A}) - g(\bm{B}) + f(0)(\bm{I}-\bm{P}_{\bm{B}})\rev{.}
\end{equation}
\rev{By the triangle inequality and the result for operator monotone functions satisfying $g(0) = 0$ we have
\begin{align*}
    \|f(\bm{A}) - f(\bm{B})_{(r)}\|_2 & \leq \|g(\bm{A}) - g(\bm{B})\|_2 + f(0)\|(\bm{I}-\bm{P}_{\bm{B}})\|_2 \\
    & \leq (1+\varepsilon) g(\lambda_{k+1}) + f(0) \\
    & \leq (1+\varepsilon)(g(\lambda_{k+1}) + f(0)) \\
    & = (1+\varepsilon) f(\lambda_{k+1}) \\
    & = (1+\varepsilon)\|f(\bm{A}) - f(\bm{A})_{(k)}\|_2,
\end{align*}
as required.}
\end{proof}


\subsection{\rev{Nyström to funNyström:} Eigenvalue guarantees}\label{section:eigenvalue_blackbox}
In this section we establish guarantees for eigenvalue estimation. In particular, \rev{using simple properties of concave functions} we show that if the eigenvalues of a SPSD matrix $\widehat{\bm{A}}$ are good approximations to the eigenvalues of $\bm{A}$, then the eigenvalues of $f(\widehat{\bm{A}})$ are even better approximations to the eigenvalues of $f(\bm{A})$. This result could be combined with results that prove eigenvalue approximation guarantees for algorithms including subspace iteration and block Krylov subspace methods \cite{MM15}.
\begin{theorem}\label{theorem:eigenvalue_black_box}
    Suppose that, for $\varepsilon \in [0,1]$, we have estimates $\widehat{\lambda}_1 \geq \widehat{\lambda}_2 \geq \ldots \geq \widehat{\lambda}_k \geq 0$ of the $k$ largest eigenvalues of $\bm{A}$ satisfying
    \begin{equation*}
        0\leq\lambda_i-\widehat{\lambda}_i \leq \varepsilon \lambda_{k+1} \quad \text{for }i = 1,\ldots,k.
    \end{equation*}
     Then for any non-decreasing concave function $f:[0,\infty) \to [0,\infty)$,
    \begin{equation*}
        0\leq f(\lambda_i)-f(\widehat{\lambda}_i) \leq \varepsilon f(\lambda_{k+1}) \quad \text{for }i = 1,\ldots,k.
    \end{equation*}
\end{theorem}

\begin{proof}
Note that by \Cref{lemma:concave_properties}(\textit{\ref*{enum:decreasing_difference}}) the function
\begin{equation*}
    g(t) = f(t) - f(t-\varepsilon \lambda_{k+1})
\end{equation*}
is decreasing. Hence, for  $i = 1,\ldots,k$ we have
\begin{align*}
    0 &\leq f(\lambda_i) - f(\widehat{\lambda}_i) \tag{\(f\) is non-decreasing} \\
    &\leq f(\lambda_i) - f(\lambda_i-\varepsilon \lambda_{k+1}) \tag{\(f\) is non-decreasing} \\
    &\leq f(\lambda_{k+1}) - f((1-\varepsilon)\lambda_{k+1}) \tag{\(g(\lambda_{i}) \leq g(\lambda_{k+1})\) since \(\lambda_i \geq \lambda_{k+1}\)} \\
    &\leq \varepsilon f(\lambda_{k+1}) \tag{\Cref{lemma:concave_properties}~(\textit{\ref*{enum:t<1}})},
\end{align*}
as required.
\end{proof}
The assumption in \Cref{theorem:eigenvalue_black_box} can be weakened to the case when we have small relative errors
\begin{equation}
    0 \leq \lambda_i - \widehat{\lambda}_i \leq \varepsilon \lambda_i.\label{eq:eig_rel_err}
\end{equation}
By the monotonicity of $f$ and \Cref{lemma:concave_properties}~(\textit{\ref*{enum:t<1}}), we have that \eqref{eq:eig_rel_err} implies
\begin{equation*}
    f(\lambda_i) - f(\widehat{\lambda}_i) \leq f(\lambda_i) - f((1-\varepsilon)\lambda_i) \leq \varepsilon f(\lambda_i).
\end{equation*}
Hence, we also have the following result. 
\begin{theorem}
Suppose that, for $\varepsilon \in [0,1]$, we have estimates $\widehat{\lambda}_1 \geq \widehat{\lambda}_2 \geq \ldots \geq \widehat{\lambda}_k \geq 0$ of the $k$ largest eigenvalues of $\bm{A}$ satisfying
    \begin{equation*}
        0\leq\lambda_i-\widehat{\lambda}_i \leq \varepsilon \lambda_{i} \quad \text{for }i = 1,\ldots,k.
    \end{equation*}
    Then for any non-decreasing concave function $f:[0,\infty) \mapsto [0,\infty)$,
    \begin{equation*}
        0\leq f(\lambda_i)-f(\widehat{\lambda}_i) \leq \varepsilon f(\lambda_{i}) \quad \text{for }i = 1,\ldots,k.
    \end{equation*}
\end{theorem}


\section{Good projections imply good Nyström approximations}\label{section:gray_box}
In this section we show that if $\bm{Q}$ is an orthonormal basis so that $(\bm{Q}\bm{Q}^T \bm{A})_{(k)}$ is a good rank $k$ approximation to $\bm{A}$, then $\widehat{\bm{A}}_{(k)}$ is a \emph{better} rank $k$ approximation to $\bm{A}$, where $\widehat{\bm{A}} = \bm{A} \bm{Q} (\bm{Q}^T \bm{A} \bm{Q})^{\dagger} \bm{Q}^T \bm{A}$ is the Nyström approximation to $\bm{A}$. 
Existing low-rank approximation literature commonly provides guarantees for the error $\|\bm{A} - (\bm{Q}\bm{Q}^T \bm{A})_{(k)}\|$, where $\bm{Q}$ is the output of some algorithm, see e.g. \cite{aineshlowrank,bakshi2023krylov,gu_subspace,iyer2018iterative,MeyerMuscoMusco:2024,MM15}.\footnote{
    We recall that $(\bm{Q} \bm{Q}^T \bm{A})_{(k)} = \bm{Q}(\bm{Q}^T \bm{A})_{(k)}$ and $(\bm{Q} \bm{Q}^T \bm{A}\bm{Q} \bm{Q}^T)_{(k)} = \bm{Q}(\bm{Q}^T \bm{A}\bm{Q})_{(k)}\bm{Q}^T$.
    Both $\bm{Q}(\bm{Q}^T \bm{A})_{(k)}$ and $\bm{Q}(\bm{Q}^T \bm{A}\bm{Q})_{(k)} \bm{Q}^T$ are preferable for computational purposes since we only have to compute the best rank $k$ approximation of a smaller matrix.
    However, in the following sections we use $(\bm{Q} \bm{Q}^T \bm{A})_{(k)}$ and $(\bm{Q} \bm{Q}^T \bm{A}\bm{Q} \bm{Q}^T)_{(k)}$, since it simplifies our notation.
} 
Hence, this result allows us to transform many known low-rank approximation guarantees into low-rank approximation guarantees for the rank $k$ truncated Nyström approximation $\widehat{\bm{A}}_{(k)}$. Further, by the results in Section~\ref{section:black_box} we therefore extend these guarantees to the funNyström approximation.

We point out that whenever $\bm{Q}$ has exactly $k$ columns, many of the results in this section would follow from \cite[Lemma 5.2]{tropp2023randomized}, which shows that $\|\bm{A} - \widehat{\bm{A}}\| \leq \|\bm{A} - \bm{Q} \bm{Q}^T \bm{A}\|$ for any unitarily invariant norm $\|\cdot\|$. However, often $\bm{Q}$ has more than $k$ columns, e.g. when $\bm{Q}$ is an orthonormal basis for a Krylov subspace, and we want to establish guarantees when we truncate $\widehat{\bm{A}}$ back to rank $k$.  Truncation is desirable when the low-rank approximation is needed for  downstream applications like data visualization or $k$-means clustering \cite{pourkamali2018randomized}.

We show that $\|\bm{A} - \widehat{\bm{A}}_{(k)}\| \leq \|\bm{A} - (\bm{Q} \bm{Q}^T \bm{A})_{(k)}\|$ for the nuclear and Frobenius norms. Perhaps surprisingly, the inequality is false in the operator norm and we provide a counterexample. Lastly, we also provide an analogous guarantee for estimating the eigenvalues of $\bm{A}$. Before doing so, we recall a standard fact about Nyström approximation used
throughout this section.

\begin{lemma}[\rev{{\cite[Lemma 1]{gittensmahoney}}}]\label{lemma:folklore}
    For any $\bm{Q}\in \mathbb{R}^{n \times \ell}$, 
    the Nyström approximation satisfies $\widehat{\bm{A}} = \bm{A} \bm{Q} (\bm{Q}^T \bm{A} \bm{Q})^{\dagger} \bm{Q}^T \bm{A} = \bm{A}^{1/2} \bm{P}_{\bm{A}^{1/2} \bm{Q}} \bm{A}^{1/2}$. 
\end{lemma}

\subsection{\rev{Projections to Nyström:} Frobenius norm guarantees}\label{section:frob_gray_box}

\rev{We first prove the following result on low-rank approximation in the Frobenius norm.}

\begin{theorem}
\label{theorem:frobenius_proj_implies_funnystrom}
Let $\bm{A} \succeq \bm{0}$ and let $\bm{Q}$ be an orthonormal basis so that, for $\epsilon \geq 0$,
    \begin{equation}\label{eq:near_optimal_rsvd}
        \|\bm{A} - (\bm{Q} \bm{Q}^T \bm{A})_{(k)}\|_F^2 \leq (1+\varepsilon) \|\bm{A}-\bm{A}_{(k)}\|_F^2.
    \end{equation}
    Then if $\widehat{\bm{A}} = \bm{A} \bm{Q} (\bm{Q}^T \bm{A} \bm{Q})^{\dagger} \bm{Q}^T \bm{A}$ we have
\begin{equation*}
    \|\bm{A}\|_F^2 - \|\widehat{\bm{A}}_{(k)}\|_F^2 \leq (1+\varepsilon) \|\bm{A}-\bm{A}_{(k)}\|_F^2. 
\end{equation*}
\end{theorem}
\Cref{theorem:frobenius_proj_implies_funnystrom} establishes that the condition needed for \Cref{theorem:frobenius_black_box} can be achieved with many low-rank approximation algorithms, including e.g. block Krylov subspace methods, sketching methods, and sampling methods \cite{aineshlowrank,bakshi2023krylov,CohenMuscoMusco:2017,MeyerMuscoMusco:2024,MM15,Sarlos:2006,woodruff_sketching}.
Further, by \Cref{lemma:schatten_difference} we know that $\|\bm{A} - \widehat{\bm{A}}_{(k)}\|_F^2 \leq \|\bm{A}\|_F^2 - \|\widehat{\bm{A}}_{(k)}\|_F^2$, which shows that if \eqref{eq:near_optimal_rsvd} is satisfied then $\|\bm{A} - \widehat{\bm{A}}_{(k)}\|_F^2 \leq (1+\varepsilon)\|\bm{A} - \bm{A}_{(k)}\|_F^2$. That is, $\widehat{\bm{A}}_{(k)}$ is a near-optimal rank $k$ approximation. 

We begin by recalling a fact about the best rank $k$ approximation to $\bm{A}$ constrained to $\range(\bm{Q})$. 
\begin{lemma}\label{lemma:constrained_best_rank_k_approximation}
    Let $\bm{Q}$ be an orthonormal basis. Then,
    \begin{equation}\label{eq:best_constrained}
        \|\bm{B} - \bm{Q} (\bm{Q}^T \bm{B})_{(k)}\|_F = \min\limits_{\bm{C} : \rank(\bm{C}) \leq k} \|\bm{B} - \bm{Q} \bm{C}\|_F^2,
    \end{equation}
    and
    \begin{equation}\label{eq:best_symmetric_constrained}
        \|\bm{B} - \bm{Q} (\bm{Q}^T \bm{B} \bm{Q})_{(k)} \bm{Q}^T\|_F = \min\limits_{\bm{C} : \rank(\bm{C}) \leq k} \|\bm{B} - \bm{Q} \bm{C} \bm{Q}^T\|_F^2.
    \end{equation}
\end{lemma}
\begin{proof}
\eqref{eq:best_constrained} was proven in \cite[Theorem 3.5]{gu_subspace} and \eqref{eq:best_symmetric_constrained} is proven in a similar fashion: \rev{Notice that for any matrix $\bm{E}$ we have by the cyclic property of the trace
\begin{equation}\label{eq:E}
    \langle \bm{B} - \bm{Q}\bm{Q}^T \bm{B}\bm{Q}\bm{Q}^T, \bm{Q}\bm{E}\bm{Q}^T \rangle = \tr(\bm{B} \bm{Q}\bm{E}^T \bm{Q}^T) - \tr(\bm{Q} \bm{Q}^T \bm{B} \bm{Q}\bm{E}^T \bm{Q}^T)=0.
\end{equation}
Let $\bm{C}$ be any matrix so that $\rank(\bm{C}) \leq k$. Then, by setting $\bm{E} = \bm{Q}^T\bm{B} \bm{Q} - \bm{C}$ in \eqref{eq:E} we obtain}
\begin{align*}
    &\langle \bm{B} - \bm{Q}\bm{Q}^T \bm{B}\bm{Q}\bm{Q}^T, \bm{Q}\bm{Q}^T \bm{B}\bm{Q}\bm{Q}^T - \bm{Q}\bm{C}\bm{Q}^T \rangle =0\rev{.} 
\end{align*}
Therefore, using the Pythagorean theorem we obtain
\begin{align*}
    &\|\bm{B} - \bm{Q} \bm{C}\bm{Q}^T \|_F^2 = \|\bm{B}- \bm{Q}\bm{Q}^T \bm{B}\bm{Q}\bm{Q}^T + \bm{Q}\bm{Q}^T \bm{B}\bm{Q}\bm{Q}^T - \bm{Q}\bm{C}\bm{Q}^T\|_F^2 \\
    =& \|\bm{B} - \bm{Q}\bm{Q}\bm{B}\bm{Q}\bm{Q}^T\|_F^2 +\|\bm{Q}^T \bm{B}\bm{Q} - \bm{C}\|_F^2.
\end{align*}
Thus, to minimize $\|\bm{B} - \bm{Q} \bm{C}\bm{Q}^T \|_F^2$ we should choose $\bm{C} = (\bm{Q}^T \bm{B} \bm{Q})_{(k)}$.
\end{proof}

With \Cref{lemma:constrained_best_rank_k_approximation} at hand, we can show that error of the rank $k$ truncated Nyström approximation is sandwiched between the error of two projection based rank $k$ approximations. 
\begin{lemma}\label{lemma:inequalities}
    Let $\bm{Q}$ be an orthonormal basis and let $\widehat{\bm{A}} = \bm{A} \bm{Q}(\bm{Q}^T \bm{A} \bm{Q})^{\dagger} \bm{Q}^T \bm{A}$ and suppose $\bm{A} \succeq \bm{0}$. Then the following holds\rev{
    \begin{align*}
        \|\bm{A} - (\bm{P}_{\bm{A} \bm{Q}} \bm{A})_{(k)}\|_F^2 &\leq \|\bm{A} - (\bm{P}_{\bm{A} \bm{Q}} \bm{A} \bm{P}_{\bm{A} \bm{Q}})_{(k)}\|_F^2 \\
        &\leq \|\bm{A} - \widehat{\bm{A}}_{(k)} \|_F^2 \\
        &\leq \|\bm{A}\|_F^2 - \|\widehat{\bm{A}}_{(k)} \|_F^2 \\
        &= \|\bm{A} - (\bm{P}_{\bm{A}^{1/2} \bm{Q}} \bm{A} \bm{P}_{\bm{A}^{1/2} \bm{Q}})_{(k)}\|_F^2 \\
        &\leq \|\bm{A} - (\bm{P}_{\bm{Q}}\bm{A})_{(k)}\|_F^2.
    \end{align*}}
\end{lemma}
\begin{proof}
The first inequality is immediate from \eqref{eq:best_constrained} since $(\bm{P}_{\bm{A} \bm{Q}} \bm{A}\bm{P}_{\bm{A} \bm{Q}})_{(k)}$ is a rank $k$ approximation whose range is contained in $\range(\bm{A}\bm{Q})$.
By a similar argument, the second inequality is immediate from \eqref{eq:best_symmetric_constrained} since $\widehat{\bm{A}}_{(k)}$ is a rank $k$ approximation whose range and co-range is contained in $\range(\bm{A} \bm{Q})$.
The third inequality is a consequence of Lemma~\ref{lemma:schatten_difference} for $p = 2$ since $\bm{A} \succeq \widehat{\bm{A}}_{(k)} \succeq \bm{0}$. 

To prove the equality, note that by \Cref{lemma:folklore} $$\widehat{\bm{A}} = \bm{A}^{1/2} \bm{P}_{\bm{A}^{1/2} \bm{Q}} \bm{A}^{1/2} = (\bm{P}_{\bm{A}^{1/2} \bm{Q}}\bm{A}^{1/2})^T(\bm{P}_{\bm{A}^{1/2} \bm{Q}}\bm{A}^{1/2}).$$ 
Hence, $\widehat{\bm{A}}_{(k)} = (\bm{P}_{\bm{A}^{1/2} \bm{Q}}\bm{A}^{1/2})_{(k)}^T(\bm{P}_{\bm{A}^{1/2} \bm{Q}}\bm{A}^{1/2})_{(k)}$.
\rev{Let $\bm{P}$ be an orthogonal projector onto the subspace spanned by the first $k$ left singular vectors of $\bm{P}_{\bm{A}^{1/2} \bm{Q}} \bm{A}^{1/2}$. Then, $\range(\bm{P}) \subseteq \range(\bm{A}^{1/2} \bm{Q})$ and $$\widehat{\bm{A}}_{(k)} = (\bm{P}\bm{P}_{\bm{A}^{1/2} \bm{Q}}\bm{A}^{1/2})^T(\bm{P}\bm{P}_{\bm{A}^{1/2} \bm{Q}}\bm{A}^{1/2}) = (\bm{P} \bm{A}^{1/2})^T(\bm{P} \bm{A}^{1/2}).$$}
Hence,
\begin{align*}
    \|\bm{A}\|_F^2 - \|\widehat{\bm{A}}_{(k)}\|_F^2 = \|\bm{A}\|_F^2 - \|(\bm{P} \bm{A}^{1/2})^T(\bm{P} \bm{A}^{1/2})\|_F^2 &= \|\bm{A}\|_F^2 - \|\bm{P} \bm{A} \bm{P}\|_F^2 \\
    &= \|\bm{A} - \bm{P} \bm{A} \bm{P}\|_F^2.
\end{align*}
Finally, noting that $\bm{P} \bm{A} \bm{P} = (\bm{P}_{\bm{A}^{1/2} \bm{Q}} \bm{A}\bm{P}_{\bm{A}^{1/2} \bm{Q}})_{(k)}$ yields the desired equality. 

For the last inequality in \Cref{lemma:inequalities}, we let $\bar{\bm{P}}$ be an orthogonal projector so that $\range(\bar{\bm{P}}) \subseteq \range(\bm{Q})$ and $\bar{\bm{P}} \bm{A} = (\bm{P}_{\bm{Q}} \bm{A})_{(k)}$.
Note that $\bm{A}^{1/2}\bar{\bm{P}} \bm{A}^{1/2}$ is a rank $k$ approximation to $\bm{A}$ whose range and co-range are both contained in $\range(\bm{A}^{1/2} \bm{Q})$.
By \eqref{eq:best_symmetric_constrained} we have that
\rev{
\begin{align*}
    \|\bm{A} - (\bm{P}_{\bm{A}^{1/2} \bm{Q}} \bm{A}\bm{P}_{\bm{A}^{1/2} \bm{Q}})_{(k)}\|_F^2 &\leq \|\bm{A} - \bm{A}^{1/2} \bar{\bm{P}}\bm{A}^{1/2}\|_F^2 \\
    &=\|(\bm{I}-\bar{\bm{P}}) \bm{A} (\bm{I}-\bar{\bm{P}})\|_F^2 \\
    & \leq \|(\bm{I}-\bar{\bm{P}}) \bm{A}\|_F^2 \\
    &= \|\bm{A} - (\bm{P}_{\bm{Q}} \bm{A})_{(k)}\|_F^2,
\end{align*}}
\rev{as required.}
\end{proof}
\begin{proof}[Proof of \Cref{theorem:frobenius_proj_implies_funnystrom}]
    The proof of our main result in this section follows immediately from Lemma~\ref{lemma:inequalities}. \rev{In particular, we have}
    \begin{equation*}
        \|\bm{A}\|_F^2 - \|\widehat{\bm{A}}_{(k)}\|_F^2 \leq \|\bm{A} - (\bm{P}_{\bm{Q}}\bm{A})_{(k)}\|_F^2 = \|\bm{A} - (\bm{Q}\bm{Q}^T \bm{A})_{(k)}\|_F^2,
    \end{equation*}
    \rev{as required.}
    \end{proof}

\begin{remark}
    \label{remark1}
    We remark on a few additional consequence of Lemma~\ref{lemma:inequalities} that may be of independent interest.
    \begin{enumerate}
        \item[1.] The lemma implies that $\|\bm{A} - (\bm{P}_{\bm{A} \bm{Q}} \bm{A})_{(k)}\|_F^2  \leq \|\bm{A} - (\bm{P}_{\bm{Q}} \bm{A})_{(k)}\|_F^2$ and that $\|\bm{A} - (\bm{P}_{\bm{A} \bm{Q}} \bm{A} \bm{P}_{\bm{A}\bm{Q}})_{(k)}\|_F^2 \leq \|\bm{A} - (\bm{P}_{ \bm{Q}} \bm{A} \bm{P}_{ \bm{Q}})_{(k)}\|_F^2$. Hence, if we approximate $\bm{A}$ via either a one-sided or two-sided projection onto $\bm{A}\bm{Q}$, the error is always better than if we simply project onto $\bm{Q}$.
        \rev{Since computing $\bm{AQ}$ is equivalent to performing one step of subspace iteration on $\bm{Q}$, we inductively establish an intuitive fact:}
        that subspace iteration monotonically \rev{decreases} Frobenius norm low-rank approximation error. Via a change of basis, a similar result is true for rectangular matrices. One can show that $\|\bm{A} - (\bm{P}_{(\bm{A}\bm{A}^T)^{p/2} \bm{Q}} \bm{A})_{(k)}\|_F^2 \leq \|\bm{A} - (\bm{P}_{\bm{Q}} \bm{A})_{(k)}\|_F^2$ for any positive integer $p$.
        \item[2.] If one has obtained an orthonormal basis $\bm{Q}$ with $\ell$ columns so that $\|\bm{A} - (\bm{P}_{\bm{Q}} \bm{A})_{(k)}\|_F \leq \epsilon$ then $\|\bm{A} - (\bm{P}_{\bm{A}\bm{Q}} \bm{A}\bm{P}_{\bm{A}\bm{Q}} )_{(k)}\|_F \leq \epsilon$. This implies that a relative error low-rank approximation guarantee for one-sided projection translates to a guarantee for two-sided projection, at the cost of at most $\ell$ extra matrix-vector products with $\bm{A}$ to form $\bm{AQ}$.
        \item[3.] Given a basis $\bm{Q}$, we require $\ell$ matrix-vector multiplications with $\bm{A}$ to either form the one-sided projection $(\bm{Q}\bm{Q}^T\bm{A})_{(k)}$ or to form the rank $k$ truncated Nyström approximation $\widehat{\bm{A}}_{(k)}$ where $\widehat{\bm{A}} = \bm{A} \bm{Q}(\bm{Q}^T \bm{A} \bm{Q})^{\dagger} \bm{Q}^T \bm{A}$. However, \rev{the} truncated Nyström approximation always provides better error in the Frobenius norm, so should be preferred.
    \end{enumerate}
\end{remark}






\subsection{\rev{Projections to Nyström:} Nuclear norm guarantees}
In this section we establish similar guarantees as in the previous section, but for the nuclear norm. Specifically, we prove the following theorem.
\begin{theorem}[\Cref{theorem:nuclear_grey_box_intro} restated]\label{theorem:nuclear_norm_gray_box}
Let $\bm{A} \succeq \bm{0}$ and let $\bm{Q}$ be an orthonormal basis so that, for $\varepsilon \geq 0$,
\begin{equation*}
    \|\bm{A} - (\bm{Q}\bm{Q}^T \bm{A})_{(k)}\|_* \leq (1+\varepsilon) \|\bm{A} - \bm{A}_{(k)}\|_*.
\end{equation*}
Then if $\widehat{\bm{A}} = \bm{A}\bm{Q} (\bm{Q}^T \bm{A} \bm{Q})^{\dagger} \bm{Q}^T \bm{A}$ we have
\begin{equation*}
    \|\bm{A} - \widehat{\bm{A}}_{(k)}\|_* \leq (1+\varepsilon) \|\bm{A} - \bm{A}_{(k)}\|_*.
\end{equation*}
\end{theorem}
\begin{proof}
Since $\bm{A} \succeq \widehat{\bm{A}}_{(k)} \succeq\bm{0}$ we know that 
\begin{equation*}
    \|\bm{A} - \widehat{\bm{A}}_{(k)} \|_* = \tr(\bm{A} - \widehat{\bm{A}}_{(k)}).
\end{equation*}
Then, since $\widehat{\bm{A}}_{(k)} = (\bm{P}_{\bm{A}^{1/2} \bm{Q}} \bm{A}^{1/2})_{(k)}^T(\bm{P}_{\bm{A}^{1/2} \bm{Q}} \bm{A}^{1/2})_{(k)}$ by \Cref{lemma:folklore}, we have  
\begin{equation*}
    \tr(\bm{A} - \widehat{\bm{A}}_{(k)}) = \|\bm{A}^{1/2}\|_F^2 - \|(\bm{P}_{\bm{A}^{1/2} \bm{Q}} \bm{A}^{1/2})_{(k)}\|_F^2 = \|\bm{A}^{1/2}-(\bm{P}_{\bm{A}^{1/2} \bm{Q}} \bm{A}^{1/2})_{(k)}\|_F^2.
\end{equation*}
Choose an orthogonal projector $\bm{P}$ \rev{that satisfies} $\range(\bm{P}) \subseteq \range(\bm{Q})$ and $\bm{P}\bm{A} = (\bm{Q}\bm{Q}^T \bm{A})_{(k)}$. Finally, by Lemma~\ref{lemma:inequalities} and Lemma~\ref{lemma:constrained_best_rank_k_approximation} we have\rev{
\begin{align*}
    \|\bm{A}^{1/2}-(\bm{P}_{\bm{A}^{1/2} \bm{Q}} \bm{A}^{1/2})_{(k)}\|_F^2 & \leq \|\bm{A}^{1/2}-(\bm{P}_{\bm{Q}} \bm{A}^{1/2})_{(k)}\|_F^2 \\
    &\leq \|\bm{A}^{1/2} - \bm{P} \bm{A}^{1/2}\|_F \\
    &= \tr(\bm{A}) - \tr(\bm{P}\bm{A} \bm{P})\\
    & = \tr(\bm{A})-\tr(\bm{P} \bm{A}) \\
    &\leq  \|\bm{A} - (\bm{Q}\bm{Q}^T \bm{A})_{(k)}\|_* \\
    & \leq (1+\varepsilon)\|\bm{A}-\bm{A}_{(k)}\|_*,
\end{align*}}
which yields the desired inequality. 
\end{proof}


\subsection{\rev{Projections to Nyström:} Operator norm guarantees}\label{section:operator_gray_box}
In this section we consider the operator norm. When $\bm{Q}$ has exactly $k$ columns, \cite{tropp2023randomized} establishes the following guarantee. 
\begin{theorem}[{\cite[Lemma 5.2]{tropp2023randomized}}]\label{theorem:spectral_gray_box}
Let $\bm{A} \succeq \bm{0}$ and let $\bm{Q} \in \mathbb{R}^{n \times k}$ be an orthonormal basis so that, for some $\varepsilon\geq 0$,
\begin{equation*}
    \|\bm{A} - \bm{Q}\bm{Q}^T \bm{A}\|_2 \leq (1+\varepsilon) \|\bm{A}-\bm{A}_{(k)}\|_2. 
\end{equation*}
Then if $\widehat{\bm{A}} = \bm{A}\bm{Q} (\bm{Q}^T \bm{A} \bm{Q})^{\dagger} \bm{Q}^T \bm{A}$ we have
\begin{equation*}
    \|\bm{A} - \widehat{\bm{A}}\|_2 \leq (1+\varepsilon) \|\bm{A}-\bm{A}_{(k)}\|_2. 
\end{equation*}
\end{theorem}
Ideally, we would extend this guarantee to the case when $\bm{Q}$ has $\ell > k$ columns, as in \Cref{theorem:frobenius_proj_implies_funnystrom,theorem:nuclear_norm_gray_box} for the Frobenius and nuclear norms. I.e., we might hope to prove that $\|\bm{A} - (\bm{Q} \bm{Q}^T \bm{A})_{(k)}\|_2 \leq (1+\varepsilon)\|\bm{A} - \bm{A}_{(k)}\|_2$ implies $\|\bm{A} - \widehat{\bm{A}}_{(k)}\|_2 \leq (1+\varepsilon)\|\bm{A} - \bm{A}_{(k)}\|_2$. Interestingly, however, we show that doing so is impossible. In particular, consider the following counterexample.
\begin{equation}\label{eq:counterexample}
    \bm{A} = \begin{bmatrix}
        9.627   &   1.538  &   -0.717 &   1.418 &    -0.309\\
        1.538   &   8.084  &   1.904 &    -1.868    & 0.573\\
        -0.717  &   1.904  &   1.353 &    -1.538    & -1.300\\
        1.418   &   -1.868    &    -1.538   &  2.534  &  0.169\\
        -0.309  &   0.573  &   -1.300  &  0.169 &   6.055
    \end{bmatrix},
    \quad
    \bm{Q} = \begin{bmatrix}
        1 & 0 & 0 \\
        0 & 1 & 0 \\
        0 & 0 & 1 \\
        0 & 0 & 0 \\
        0 & 0 & 0
    \end{bmatrix},
\end{equation}
\rev{where $\bm{A}$ was found by a random search over SPSD matrices.} For these matrices, $\|\bm{A} - (\bm{Q} \bm{Q}^T \bm{A})_{(2)}\|_2 \approx (1+2.59 \times 10^{-8}) \|\bm{A}-\bm{A}_{(2)}\|_2$ whereas $\|\bm{A} - \widehat{\bm{A}}_{(2)}\|_2 \approx (1+5.75 \times 10^{-3}) \|\bm{A}-\bm{A}_{(2)}\|_2$, so $(\bm{Q}\bm{Q}^T \bm{A})_{(k)}$ is a better rank $k$ approximation to $\bm{A}$ compared to $\widehat{\bm{A}}_{(k)}$. As guaranteed by \Cref{theorem:spectral_gray_box}, we do at least have that $3.75 \approx \|\bm{A} - \widehat{\bm{A}}\|_2 < \|\bm{A} - \bm{Q} \bm{Q}^T \bm{A}\|_2 \approx 6.24$. Via the same counterexample, we also have the following.

\begin{remark}
    In Section~\ref{section:frob_gray_box} we showed that Frobenius norm low-rank approximation error decreases monotonically in the number of subspace iterations (see \Cref{remark1}). The same is not true in the operator norm. To see this, let $\bm{A}$ and $\bm{Q}$ be as in \eqref{eq:counterexample}. We can check that
    \begin{equation*}
        \|\bm{A} - (\bm{P}_{\bm{Q}} \bm{A})_{(2)}\|_2 \approx 6.449 < 6.455 \approx \|\bm{A} - (\bm{P}_{\bm{A}\bm{Q}} \bm{A})_{(2)}\|_2.
    \end{equation*}
\end{remark}

\subsection{\rev{Projections to Nyström:} Eigenvalue guarantee\rev{s}}
We conclude the theoretical part of the paper by noting a guarantee for eigenvalue estimation. Specifically,
if we have a basis $\bm{Q}$ so that \rev{the} top $k$ singular values of $\bm{Q} \bm{Q}^T \bm{A}$ are estimates of the eigenvalues of $\bm{A}$, then the eigenvalues of the Nyström approximation $\widehat{\bm{A}} = \bm{A} \bm{Q} (\bm{Q}^T \bm{A} \bm{Q})^{\dagger} \bm{Q}^T \bm{A}$ can only be better estimates. 
\begin{theorem}
Let $\bm{A} \succeq \bm{0}$ and let \(\bm{Q}\) be an orthonormal basis so that \(\lambda_i - \varepsilon_i \leq \sigma_{i}(\bm{Q}^T \bm{A}) \leq \lambda_i\) for \(i=1,...,k\) and for  $\varepsilon_1,\ldots,\varepsilon_k \geq 0$.
Then if $\widehat{\lambda}_i$ is the $i^\text{th}$ \rev{largest} eigenvalue of \(\widehat{\bm{A}} = \bm{A}\bm{Q} (\bm{Q}^T \bm{A} \bm{Q})^{\dagger} \bm{Q}^T \bm{A}\), we have \(\lambda_i - \varepsilon_i \leq \widehat{\lambda}_i \leq \lambda_i\) for \(i=1,...,k\).
\end{theorem}
\begin{proof}
Notice that, by \Cref{lemma:folklore}, we have
\begin{equation*}
    \bm{Q}^T\widehat{\bm{A}}
    = (\bm{Q}^T \bm{A}^{1/2} \bm{P}_{\bm{A}^{1/2} \bm{Q}} )\bm{A}^{1/2} 
    = \bm{Q}^T\bm{A}.
\end{equation*}
Therefore, by applying a standard singular value inequality \cite[p.452]{matrixanalysis} we have
\[
    \widehat{\lambda_i}
    = \sigma_i(\widehat{\bm{A}})
    \geq \sigma_i(\bm{Q}^T\widehat{\bm{A}})
    = \sigma_i(\bm{Q}^T\bm{A})
    \geq \lambda_i - \varepsilon_i.
\]
We complete the proof by noting that \(\widehat{\lambda_i}\leq \lambda_i\) because \(\widehat{\bm{A}} \preceq \bm{A}\) \cite[Corollary 7.7.4 (c)]{matrixanalysis}.
\end{proof}

\section{Numerical experiments}
In this section, we numerically verify our theoretical results. Experiments were performed on a MacBook Pro in MATLAB (v. 2020a) and scripts to reproduce our figures are available at \url{https://github.com/davpersson/funNystrom-v2}. 

In all our experiments, we begin with computing an orthonormal basis $\bm{Q}$ with $\ell \geq k$ columns. 
In Sections~\ref{section:numex1}-\ref{section:numex3} we outline three different algorithms for doing so. 
Next, using the orthonormal basis $\bm{Q}$, we construct a Nyström approximation $\widehat{\bm{A}}$ as defined in \eqref{eq:nystrom} and the projection based approximation $\bm{Q}\bm{Q}^T \bm{A}$. 
Note that once $\bm{Q}$ is computed, constructing the two approximations comes at the same computational cost. 
Then, we truncate $\widehat{\bm{A}}$ and $\bm{Q}\bm{Q}^T \bm{A}$ to rank $k$ to obtain $\widehat{\bm{A}}_{(k)}$ and $(\bm{Q}\bm{Q}^T \bm{A})_{(k)}$. 
Finally, we compare the following quantities
\begin{align*}
    &\varepsilon_{\text{projection}} = \frac{\|\bm{A}-(\bm{Q}\bm{Q}^T \bm{A})_{(k)}\|}{\|\bm{A}-\bm{A}_{(k)}\|} -1;\\
    &\varepsilon_{\text{Nyström}} = \frac{\|\bm{A}-\widehat{\bm{A}}_{(k)}\|}{\|\bm{A}-\bm{A}_{(k)}\|} -1;\\
    &\varepsilon_{\text{funNyström}} = \frac{\|f(\bm{A})-f(\widehat{\bm{A}}_{(k)})\|}{\|f(\bm{A})-f(\bm{A}_{(k)})\|} -1,
\end{align*}
where $\|\cdot\| = \|\cdot\|_*, \|\cdot\|_F$ or $\|\cdot\|_2$. For comparing accuracy in estimating eigenvalues we use the alternative metrics:
\begin{align*}
    &\varepsilon_{\text{projection}} = \max\limits_{i=1,\ldots,k} \left\{\frac{\lambda_i - \sigma_i(\bm{Q}^T \bm{A})}{\lambda_i}\right\};\\
    &\varepsilon_{\text{Nyström}} = \max\limits_{i=1,\ldots,k} \left\{\frac{\lambda_i - \widehat{\lambda}_i}{\lambda_i}\right\};\\
    &\varepsilon_{\text{funNyström}} = \max\limits_{i=1,\ldots,k} \left\{\frac{f(\lambda_i) - f(\widehat{\lambda}_i)}{f(\lambda_i)}\right\}.
\end{align*}
Our theory suggests that, for the Frobenius norm, nuclear norm, or for eigenvalue estimation, $\varepsilon_{\text{projection}} \geq \varepsilon_{\text{Nyström}} \geq \varepsilon_{\text{funNyström}}$. For the operator norm, we expect the second inequality to hold, but we have shown a counterexample to the first in \Cref{section:operator_gray_box}. 
However, in our experiments we generally observe that $\varepsilon_{\text{projection}} \geq \varepsilon_{\text{Nyström}}$ even when $\|\cdot\| = \|\cdot\|_2$.

\subsection{Column subset selection}\label{section:numex1}
In this experiment we compute the orthonormal basis $\bm{Q}$ using the randomly pivoted Cholesky algorithm \cite[Algorithm 2.1]{chen2022randomly}. 
In this setting, $\bm{Q} = \begin{bmatrix} \bm{e}_{i_1} & \ldots & \bm{e}_{i_{\ell}} \end{bmatrix}$ where $\{i_1,\ldots,i_{\ell}
\} \subseteq \{1,\ldots,n\}$ is an index set returned by the algorithm, and $\bm{e}_i$ is the $i^{\text{th}}$ standard basis vector. 
We set $k = 10$ and $\ell = k + q$ for $q = 0,\ldots,6$. 
We let $\bm{A} \in \mathbb{R}^{1000\times 1000}$ be defined by
\begin{equation*}
    \bm{A}_{ij} = \left(\left(\frac{i}{1000}\right)^{10} + \left(\frac{j}{1000}\right)^{10}\right)^{\frac{1}{10}},
\end{equation*}
where $\bm{A}_{ij}$ denotes the $(i,j)$-entry of $\bm{A}$. This example is inspired by the numerical experiments on column subset selection in \cite{cortinovisfrobenius}.
We set the matrix function to be $f(x) = \frac{x}{x+1}$, which is operator monotone. 
The results are presented in \Cref{fig:fig1}, which shows that $\varepsilon_{\text{projection}} \geq \varepsilon_{\text{Nyström}} \geq \varepsilon_{\text{funNyström}}$ for all norms and $q$ we consider, which confirms our theoretical results.
\begin{figure}[h]
\begin{subfigure}{.49\textwidth}
  \centering
  \includegraphics[width=.9\linewidth]{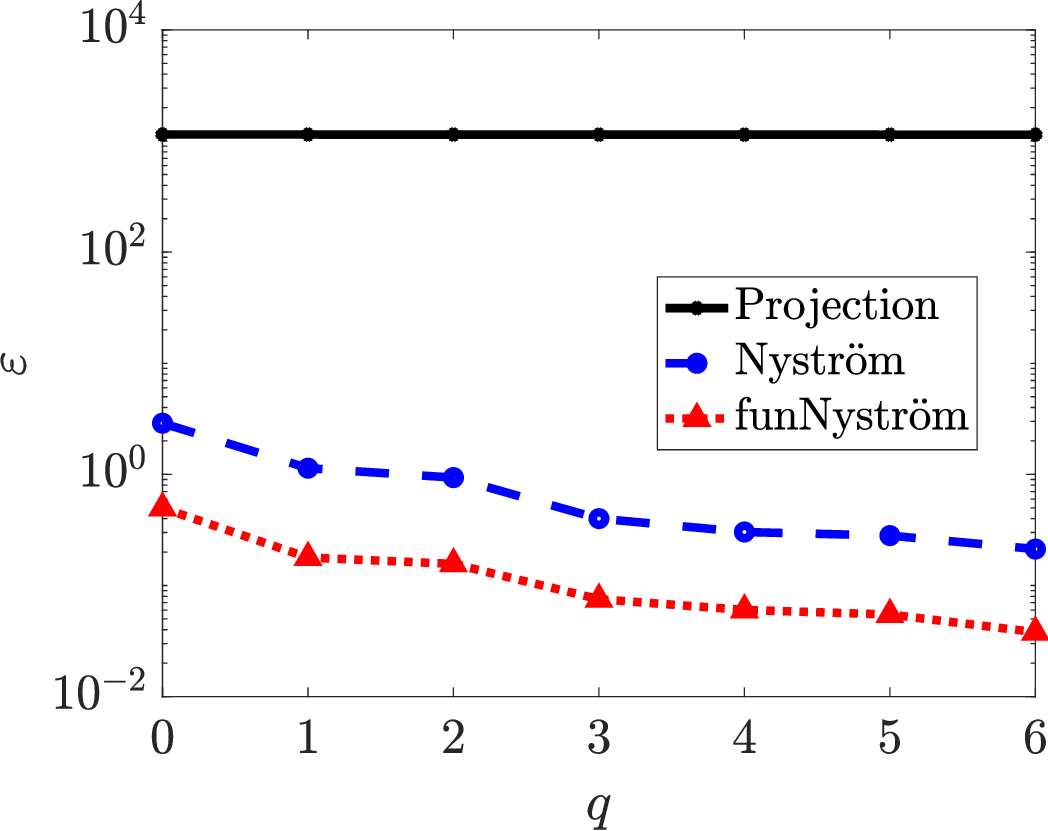}  
  \caption{Nuclear norm}
\end{subfigure}
\begin{subfigure}{.49\textwidth}
  \centering
  \includegraphics[width=.9\linewidth]{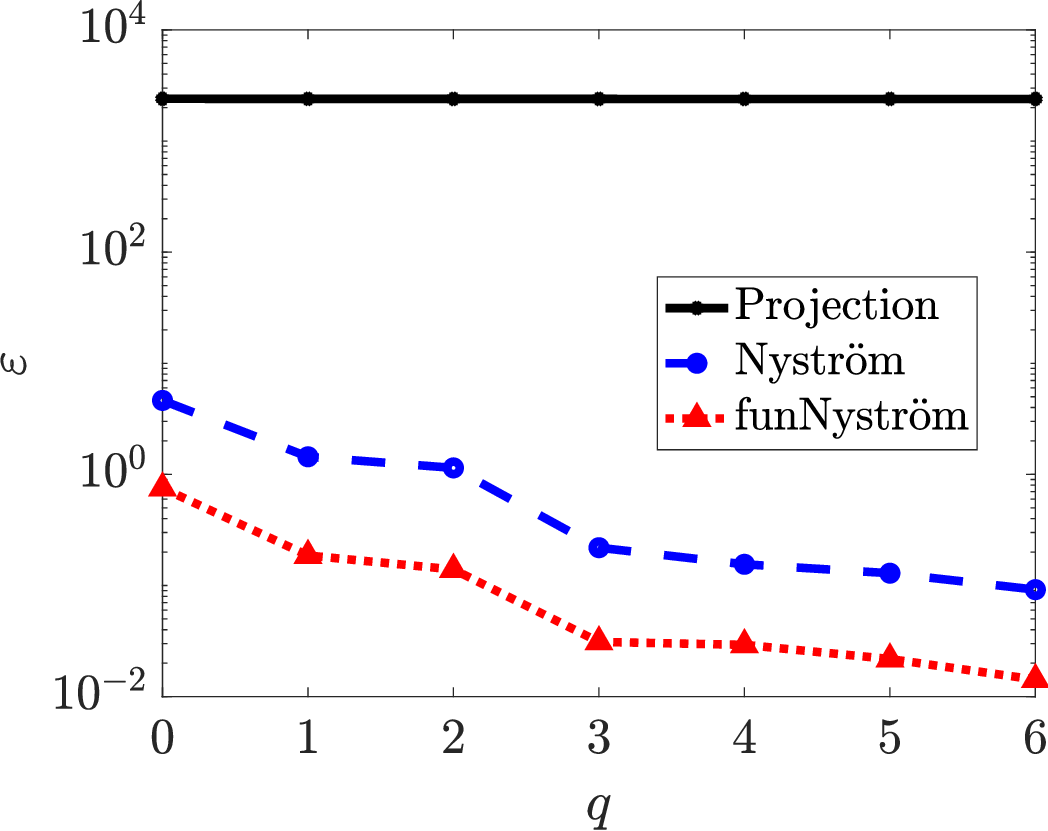}  
  \caption{Frobenius norm}
\end{subfigure}
\begin{subfigure}{.49\textwidth}
  \centering
  \includegraphics[width=.9\linewidth]{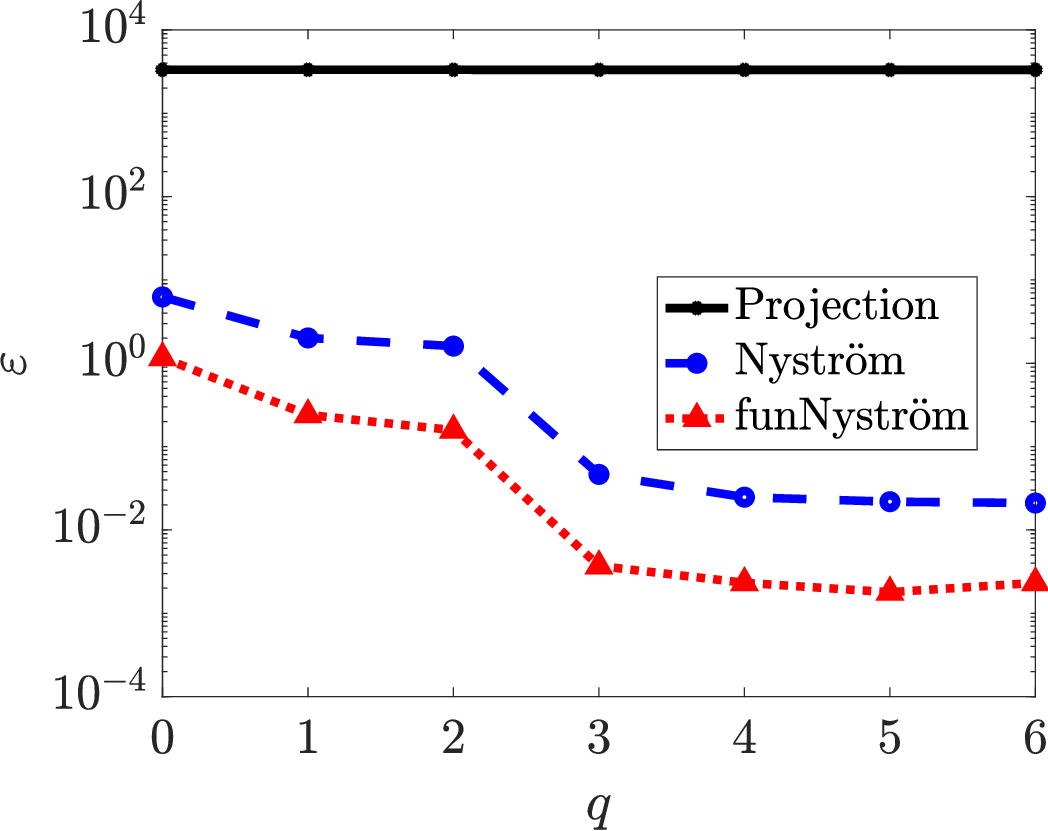}  
  \caption{Operator norm}
\end{subfigure}
\begin{subfigure}{.49\textwidth}
  \centering
  \includegraphics[width=.9\linewidth]{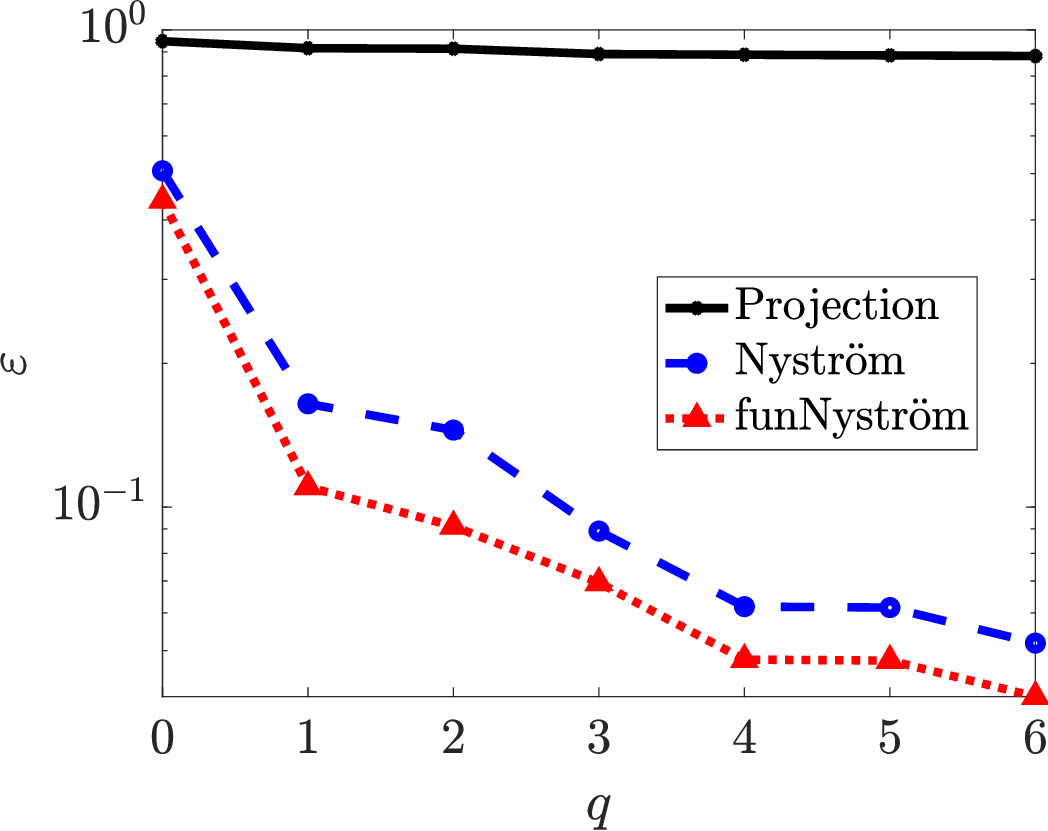}  
  \caption{Eigenvalue estimates}
\end{subfigure}
\caption{Comparing $\varepsilon_{\text{projection}}, \varepsilon_{\text{Nyström}},$ and $\varepsilon_{\text{funNyström}}$ for column subset selection. Note that $\varepsilon_{\text{projection}}$ is significantly worse than the $\varepsilon_{\text{Nyström}}$ and $\varepsilon_{\text{funNyström}}$ since the orthogonal projection $\bm{Q}\bm{Q}^T$ zeros out all except $\ell$ rows of $\bm{A}$.
In contrast, the Nystr\"om approximation $\widehat{\bm{A}} = \bm{A}^{1/2} \bm{P}_{\bm{A}^{1/2} \bm{Q}} \bm{A}^{1/2}$ effectively performs half a step of subspace iteration on $\bm{Q}$, giving a better approximation.
}
\label{fig:fig1}
\end{figure}
\subsection{Krylov iteration}
In this experiment we set $\bm{Q}$ to be an orthonormal basis for the Krylov subspace $\range(\begin{bmatrix} \bm{\Omega} & \bm{A} \bm{\Omega} & \ldots & \bm{A}^{q} \bm{\Omega}\end{bmatrix})$ where $\bm{\Omega}$ is a random $3000 \times k$ matrix whose entries are independent identically distributed  Gaussian random variables with mean $0$ and variance $1$. We set $k = 10$ and vary $q = 0,1,\ldots,6$. We let $\bm{A}$ be a $3000 \times 3000 $ SPSD matrix whose eigenvalues are $\lambda_i = i^{-1}$ for $i = 1,\ldots,n$. We set the matrix function to be $f(x) = \log(1+x)$, which is operator monotone. The results are presented in \Cref{fig:fig2}.
Again, for all norms and all choices of $q$, we see that $\varepsilon_{\text{projection}} \geq \varepsilon_{\text{Nyström}} \geq \varepsilon_{\text{funNyström}}$.
\begin{figure}[h]
\begin{subfigure}{.49\textwidth}
  \centering
  \includegraphics[width=.9\linewidth]{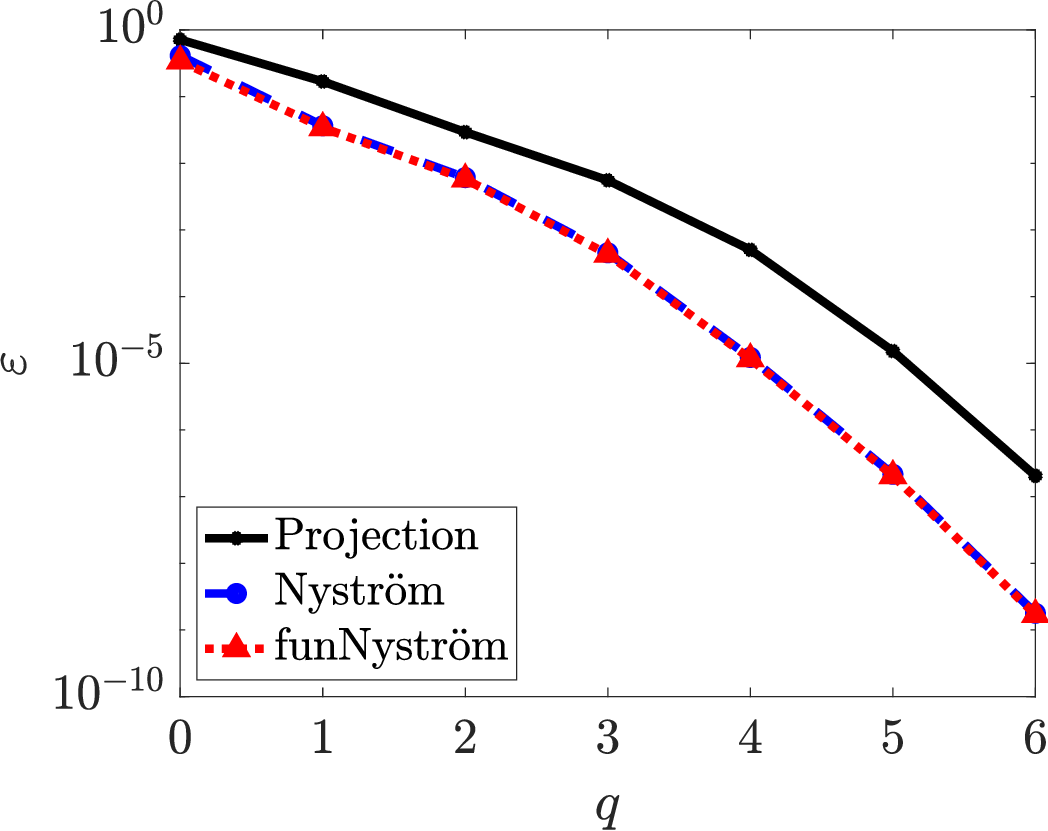}  
  \caption{Nuclear norm}
\end{subfigure}
\begin{subfigure}{.49\textwidth}
  \centering
  \includegraphics[width=.9\linewidth]{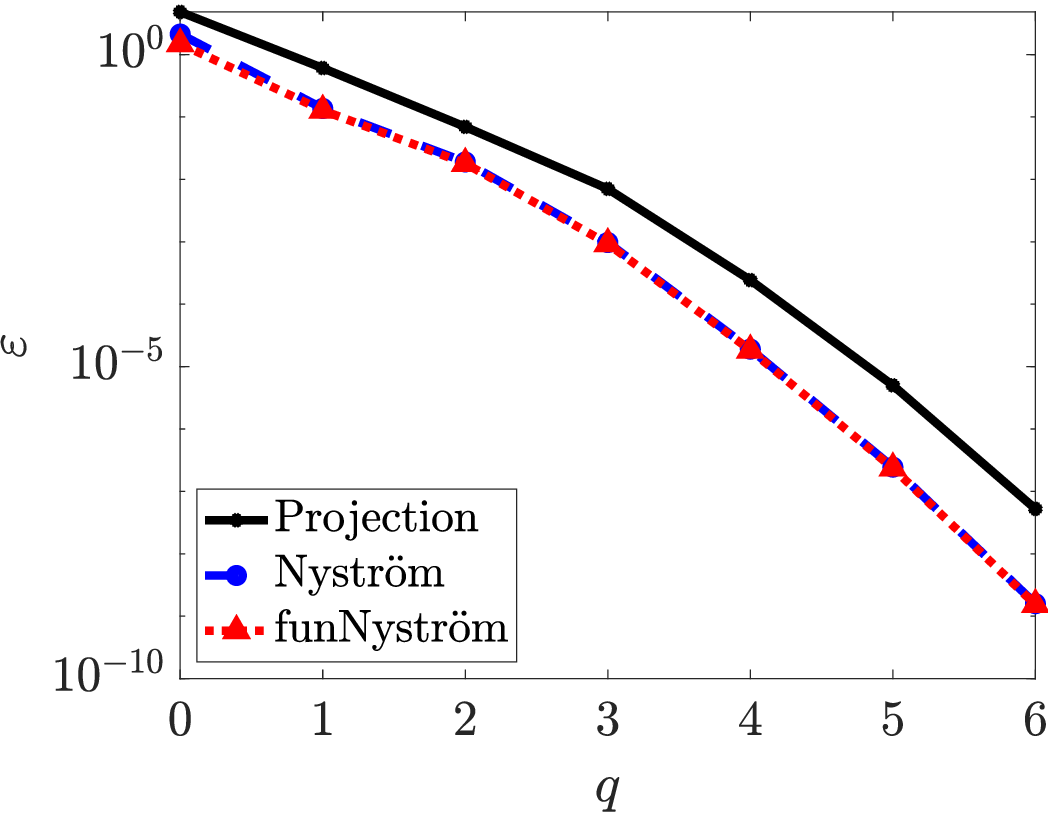}  
  \caption{Frobenius norm}
\end{subfigure}
\begin{subfigure}{.49\textwidth}
  \centering
  \includegraphics[width=.9\linewidth]{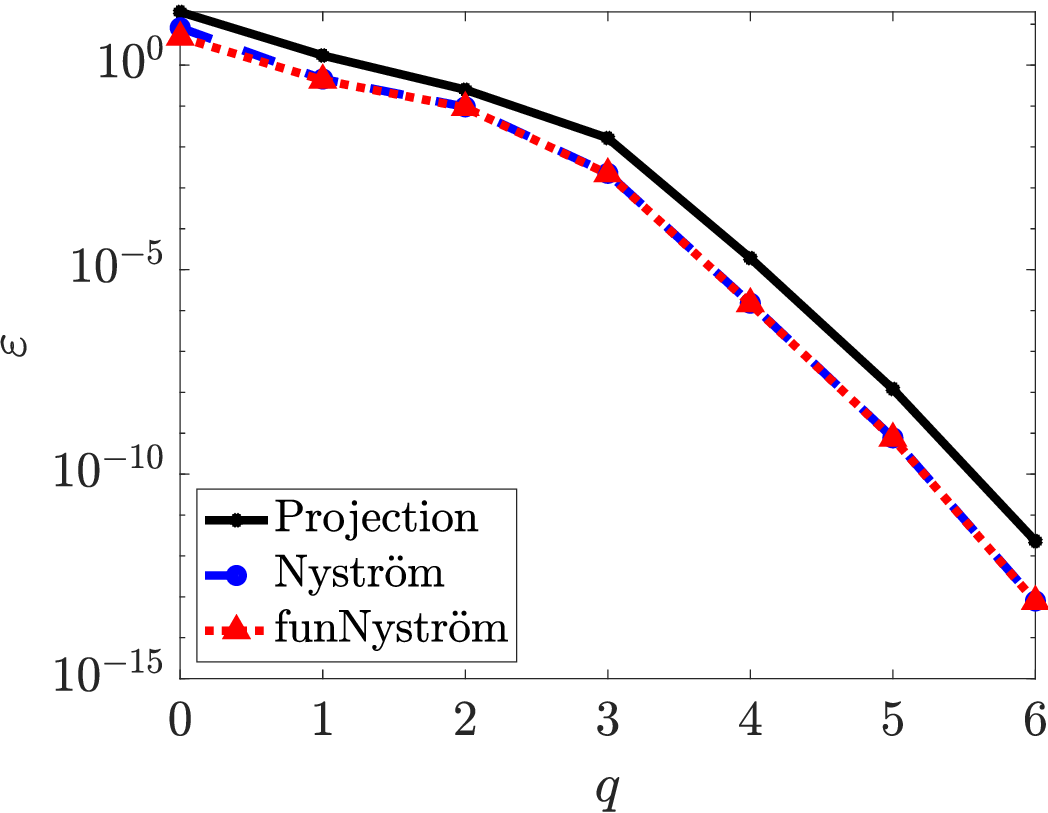}  
  \caption{Operator norm}
\end{subfigure}
\begin{subfigure}{.49\textwidth}
  \centering
  \includegraphics[width=.9\linewidth]{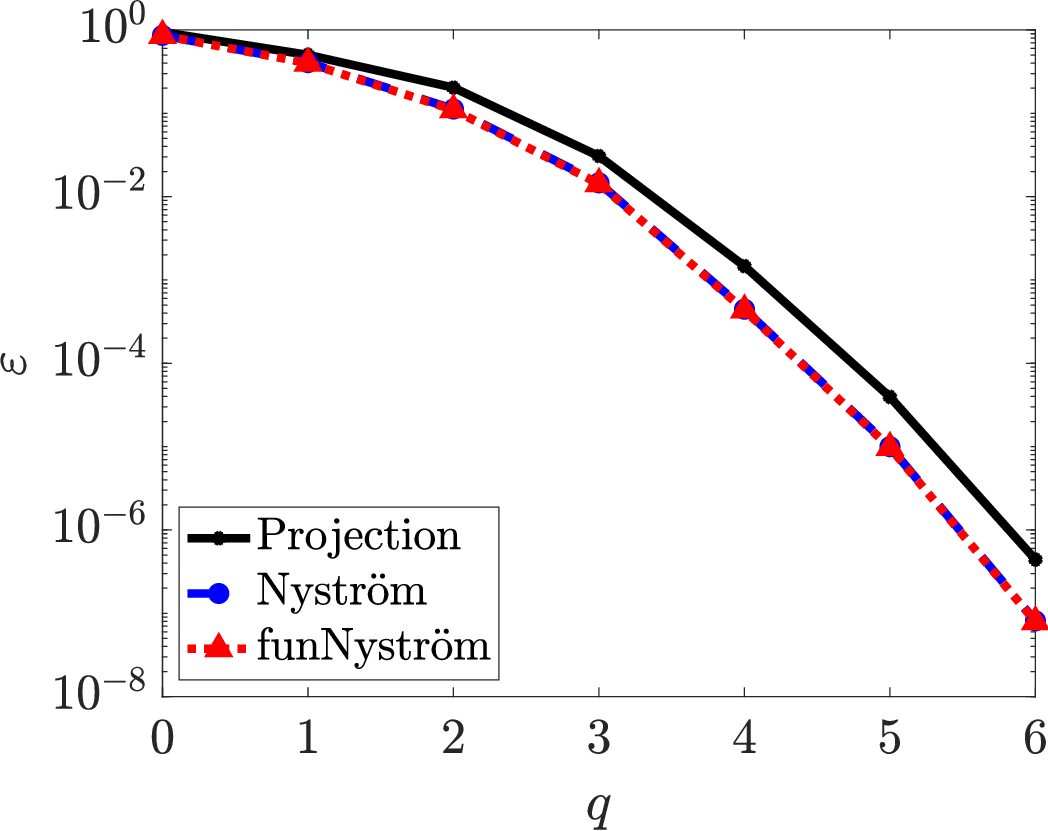}  
  \caption{Eigenvalue estimates}
\end{subfigure}
\caption{Comparing $\varepsilon_{\text{projection}}, \varepsilon_{\text{Nyström}},$ and $\varepsilon_{\text{funNyström}}$ for Krylov iteration.}
\label{fig:fig2}
\end{figure}
\subsection{Subspace iteration}\label{section:numex3}
Finally, we set $\bm{Q}$ to be an orthonormal basis for $\range(\bm{A}^q \bm{\Omega})$ where $\bm{\Omega}$ is a random $3000 \times k$ matrix whose entries are independent identically distributed  Gaussian random variables with mean $0$ and variance $1$.
We set $k = 10$ and vary $q = 0,1,\ldots,6$. 
We let $\bm{A}$ be a $3000 \times 3000 $ SPSD matrix whose eigenvalues are $\lambda_i = e^{-i}$ for $i = 1,\ldots,n$. 
We set the matrix function to be $f(x) = x^{1/2}$, which is operator monotone. 
The results are presented in \Cref{fig:fig3}. Once again, for all norms and all choices of $q$, we see that $\varepsilon_{\text{projection}} \geq \varepsilon_{\text{Nyström}} \geq \varepsilon_{\text{funNyström}}$.

\begin{figure}[h]
\begin{subfigure}{.49\textwidth}
  \centering
  \includegraphics[width=.9\linewidth]{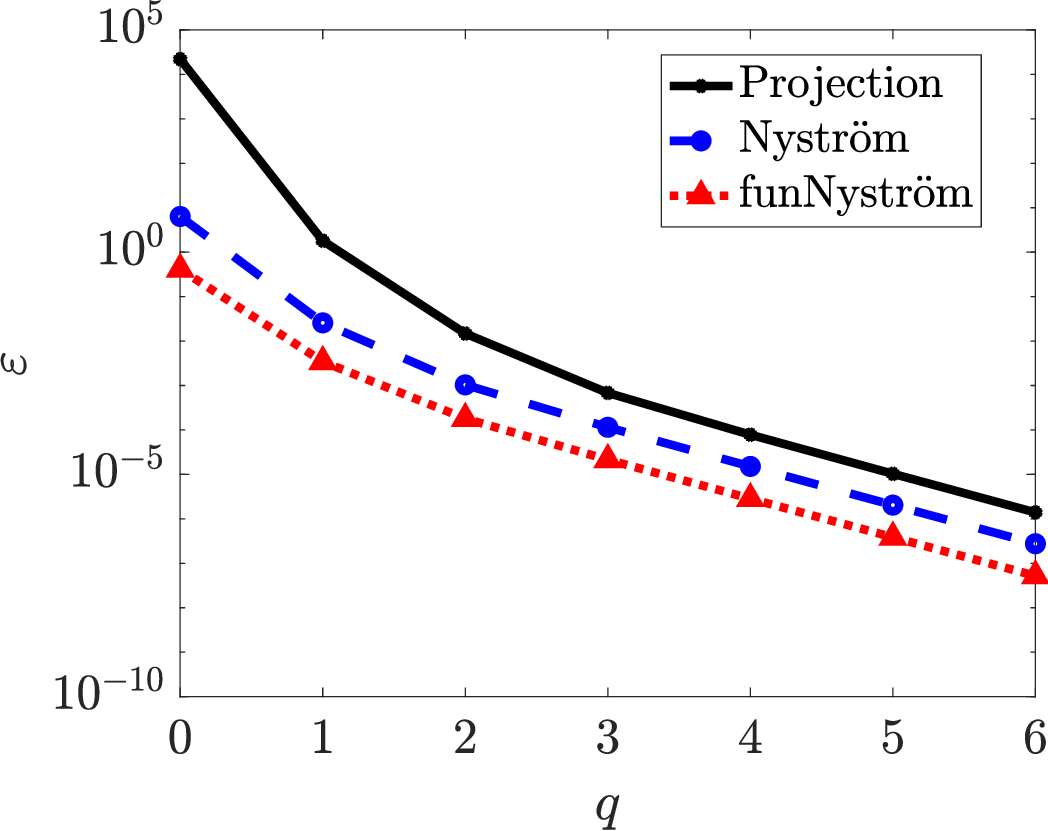}  
  \caption{Nuclear norm}
\end{subfigure}
\begin{subfigure}{.49\textwidth}
  \centering
  \includegraphics[width=.9\linewidth]{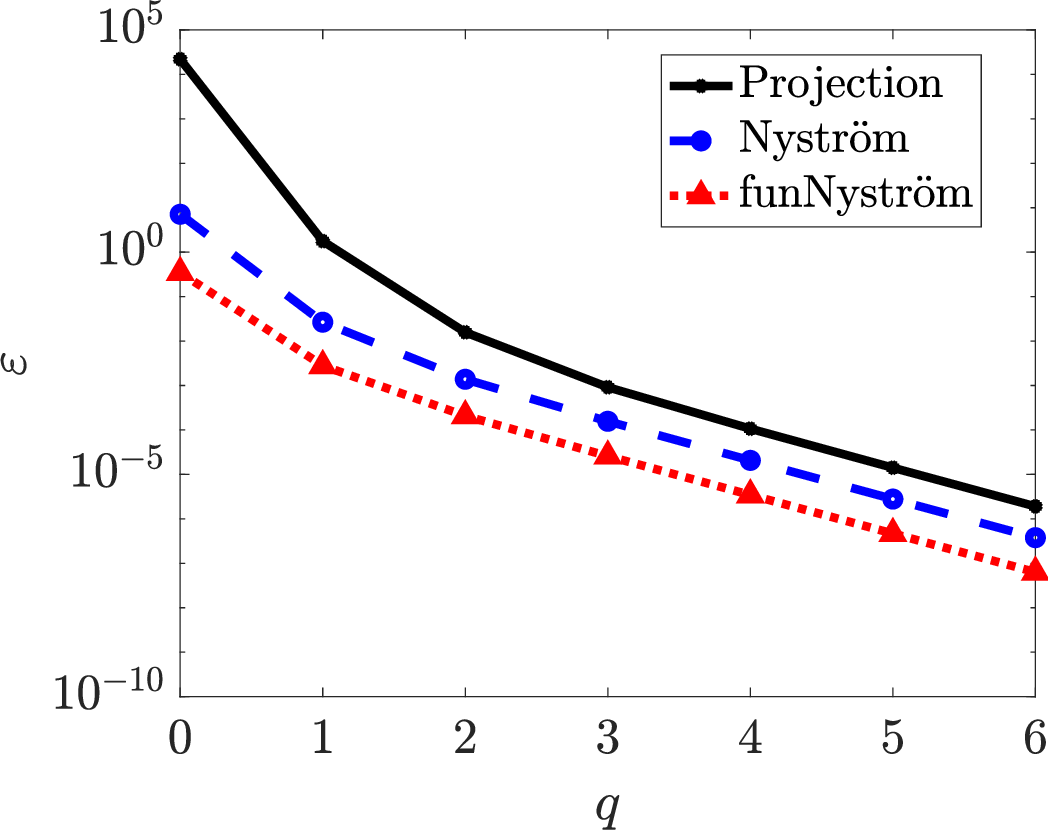}  
  \caption{Frobenius norm}
\end{subfigure}
\begin{subfigure}{.49\textwidth}
  \centering
  \includegraphics[width=.9\linewidth]{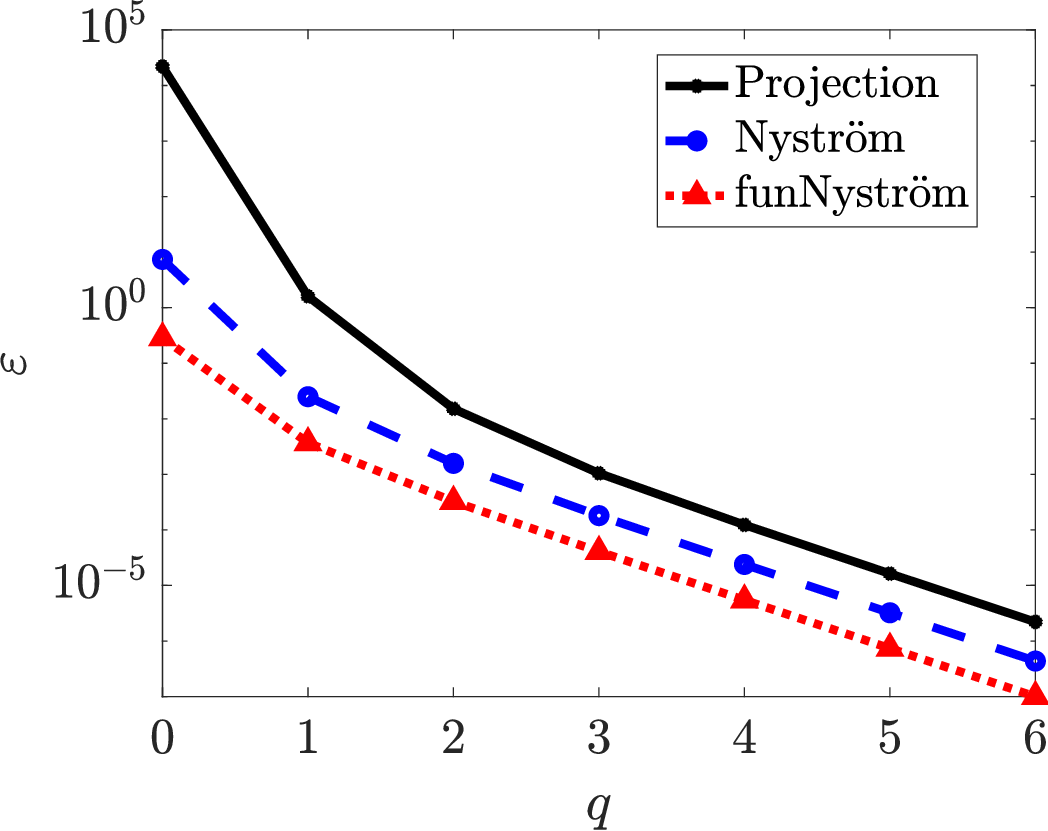}  
  \caption{Operator norm}
\end{subfigure}
\begin{subfigure}{.49\textwidth}
  \centering
  \includegraphics[width=.9\linewidth]{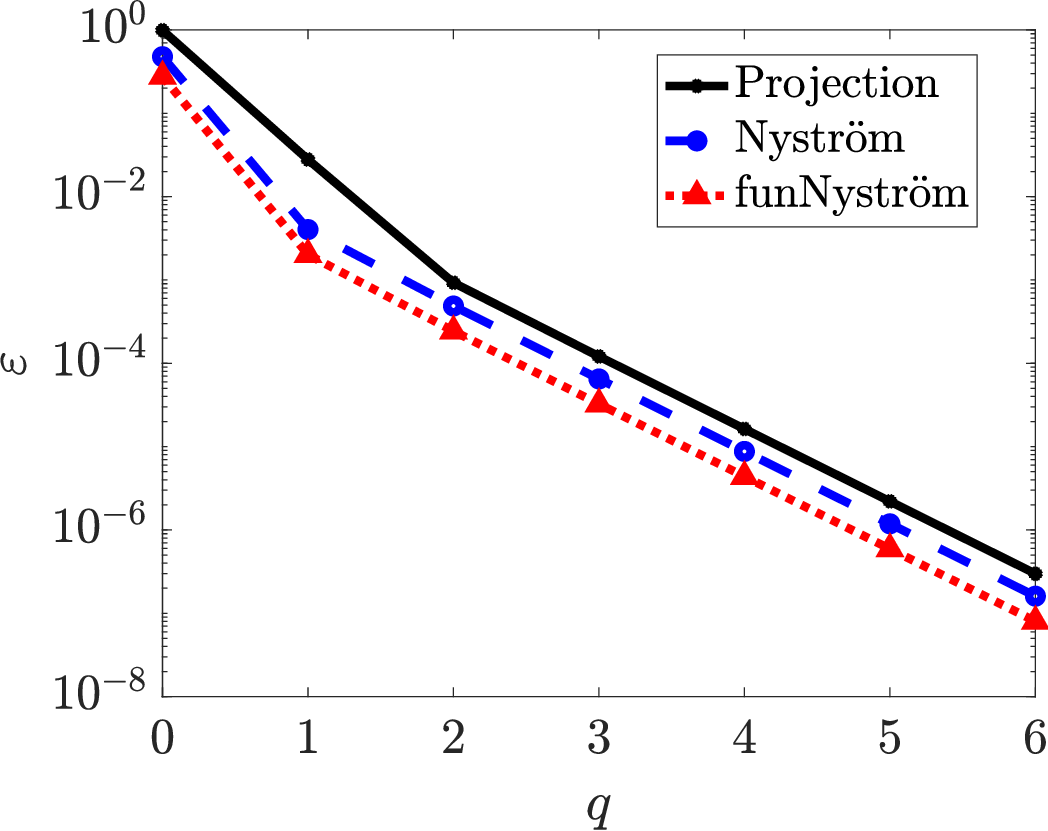}  
  \caption{Eigenvalue estimates}
\end{subfigure}
\caption{Comparing $\varepsilon_{\text{projection}}, \varepsilon_{\text{Nyström}},$ and $\varepsilon_{\text{funNyström}}$ for subspace iteration.}
\label{fig:fig3}
\end{figure}

\section{\rev{Examples showing the necessity of our assumptions}}
In this section, we give examples that demonstrate that the assumptions in our results in \Cref{section:black_box} are necessary.
Recall that in \Cref{section:black_box}, $\widehat{\bm{A}}$ is a SPSD rank $k$ approximation of $\bm{A}$, not necessarily the Nyström approximation.
We define
\begin{align*}
    \varepsilon_{\text{original}} &= \frac{\|\bm{A}-\widehat{\bm{A}}\|}{\|\bm{A}-\bm{A}_{(k)}\|} - 1; &
    \varepsilon_{\text{function}} &= \frac{\|f(\bm{A})-f(\widehat{\bm{A}})\|}{\|f(\bm{A})-f(\bm{A}_{(k)})\|} - 1,
\end{align*}
where $\|\cdot\| = \|\cdot\|_*,\|\cdot\|_F$ or $\|\cdot\|_2$. The results in \Cref{section:black_box} show that if $(i)$ $\bm{A} \succeq \widehat{\bm{A}}$, and $(ii)$ $f$ is a continuous operator monotone function with $f(0) \geq 0$, then $\varepsilon_{\text{function}} \leq \varepsilon_{\text{original}}$.
We might ask if we really have to assume that $\bm{A} \succeq \widehat{\bm{A}}$ or if it is sufficient to only assume that $\widehat{\bm{A}}$ is SPSD.
Similarly, we might ask if $f$ has to be operator monotone, or if it is sufficient to only assume that $f$ is non-decreasing, continuous, and concave.

We provide counter examples to show that, in most cases, the answer to these questions is no: our current assumptions are necessary. 
Specifically, we show instances where $\varepsilon_{\text{function}} > \varepsilon_{\text{original}}$ for specific norms when the assumptions are relaxed. 
Our results are summarized in 
Table~\ref{table:summ} in \Cref{section:intro}. 

\begin{example}[Operator monotone is necessary for nuclear norm]
\label{eg:mono-needed-for-nuc}
Let \(\bm{A} = \sbmat{ 1.1 \\ & 0.1 }\) and \(\widehat{\bm{A}} = \bm{v}\bm{v}^T\) where \(\bm{v}=\sbmat{1 \\ 0.095}\).
With \(k=1\) and \(f(x)=\min\{1,x\}\), we have that \(\varepsilon_{\text{function}} > 1.158\, \varepsilon_{\text{original}}\) in the nuclear norm.
For this example, \(\bm{A} \succeq \widehat{\bm{A}}\) and $f$ is non-decreasing and concave, but not operator monotone.
\end{example}

\begin{example}[\(\widehat{\bm{A}} \preceq \bm{A}\) or operator monotone is necessary for operator norm]
\label{eg:mono-needed-for-spectral}
Let \(\bm{A} = \sbmat{ 1.01 \\ & 0.01 }\) and \(\widehat{\bm{A}} = \bm{v}\bm{v}^T\) where \(\bm{v}=\sbmat{1.01 \\ 0.01}\).
With \(k=1\) and \(f(x)=\min\{1,x\}\), we have \(\varepsilon_{\text{function}} > 1.402\, \varepsilon_{\text{original}}\) in the operator norm.
For this example, \(\bm{A} \cancel{\succeq} \widehat{\bm{A}}\) and $f$ is non-decreasing and concave, but not operator monotone.
\end{example}
For the two examples above, we can confirm that $f(x) = \min\{1,x\}$ is not operator monotone because it is not differentiable, which violates \Cref{lemma:opmon_properties}~(\textit{\ref*{eqnum:opmon_implies_smooth}}).

\begin{example}[\(\widehat{\bm{A}} \preceq \bm{A}\)  is necessary for Frobenius and nuclear norms.]
\label{eg:psd-needed-for-frob-nuc}
Let \(\bm{A} = \sbmat{ 1 \\ & 1 \\ & & 0 }\) and \(\widehat{\bm{A}} = \bm{v}\bm{v}^T\bm{A}\bm{v}\bm{v}^T\) where \(\bm{v} = \frac1{\sqrt2}\sbmat{ 0 \\ 1 \\ 1 }\).
Then, with \(k=1\) and \(f(x) = \sqrt{x}\) we have \(\varepsilon_{\text{function}} > 1.095\varepsilon_{\text{original}}\) in the nuclear norm and \(\varepsilon_{\text{function}} > 1.049\varepsilon_{\text{original}}\) in the Frobenius norm.
For this example, \(f\) is operator monotone, but it can be checked that \(\bm{A}\cancel{\succeq}\widehat{\bm{A}}\). 
\end{example}

Despite the above counterexamples, gaps still remain in our understanding of what assumptions are necessary for a good low-rank approximation to $\bm{A}$ to imply a good low-rank approximation for $f(\bm{A})$. For example, we have not found a counter example for the operator and Frobenius norms when $\bm{A} \succeq \widehat{\bm{A}}$ and $f$ is non-decreasing and concave, but not operator monotone. Additionally, we emphasize that our counterexampless do not rule our ``approximate'' versions of our results on funNyström (\Cref{theorem:frobenius_black_box,theorem:nuclear_black_box,theorem:spectral_black_box}). For example, while $\varepsilon_{\text{function}}$ can be larger than $\varepsilon_{\text{original}}$, it might be possible to prove that $\varepsilon_{\text{function}} \leq  2\cdot \varepsilon_{\text{original}}$ under relaxed assumptions.

\section{Conclusion}
In this work, we significantly generalize a result of Persson and Kressner \cite{persson2022randomized} by showing that near-optimal Nyström approximations imply near-optimal funNyström approximations, independently of how the Nyström approximation was obtained. 
This implies that any method that computes good Nyström approximations can be used to obtain good low-rank approximations of operator monotone matrix functions.
We have also shown that good low-rank projections imply good Nyström approximations. This allows us to translate low-rank approximation guarantees for many common algorithms into guarantees for  Nyström approximation, and thus for funNyström approximation. 

\begin{paragraph}{Acknowledgments} 
We thank Daniel Kressner for helpful comments on this work.
\end{paragraph}


\bibliographystyle{siam}
\bibliography{bibliography}

\end{document}